\documentclass[twoside]{article}
\usepackage{amsmath,amssymb,amscd,amsthm,alltt,amsfonts,array}
\usepackage{mathrsfs}
\usepackage[english]{babel}
\usepackage{latexsym}
\usepackage{euscript}
\usepackage{graphicx}
\usepackage[shortlabels]{enumitem}

\newcommand{\salt}{\vspace*{2.5mm}}     
\newcommand{\ap}{@}

\usepackage[inner=4cm,outer=4cm]{geometry}

\def\boxit#1{\vbox{\hrule\hbox{\vrule\kern2pt
 \vbox{\kern4pt#1\kern2pt}\kern2pt\vrule}\hrule}}
\def\qed{\hfill
 $\hskip0.3cm
 \boxit{\hsize 2pt \vsize 10pt}$\bigskip\noindent}

\def\C{{\mathchoice {\setbox0=\hbox{$\displaystyle\rm C$}\hbox{\hbox
to0pt{\kern0.4\wd0\vrule height0.9\ht0\hss}\box0}}
{\setbox0=\hbox{$\textstyle\rm C$}\hbox{\hbox
to0pt{\kern0.4\wd0\vrule height0.9\ht0\hss}\box0}}
{\setbox0=\hbox{$\scriptstyle\rm C$}\hbox{\hbox
to0pt{\kern0.4\wd0\vrule height0.9\ht0\hss}\box0}}
{\setbox0=\hbox{$\scriptscriptstyle\rm C$}\hbox{\hbox
to0pt{\kern0.4\wd0\vrule height0.9\ht0\hss}\box0}}}}
\def\Q{{\mathchoice {\setbox0=\hbox{$\displaystyle\rm
Q$}\hbox{\raise
0.15\ht0\hbox to0pt{\kern0.4\wd0\vrule height0.8\ht0\hss}\box0}}
{\setbox0=\hbox{$\textstyle\rm Q$}\hbox{\raise
0.15\ht0\hbox to0pt{\kern0.4\wd0\vrule height0.8\ht0\hss}\box0}}
{\setbox0=\hbox{$\scriptstyle\rm Q$}\hbox{\raise
0.15\ht0\hbox to0pt{\kern0.4\wd0\vrule height0.7\ht0\hss}\box0}}
{\setbox0=\hbox{$\scriptscriptstyle\rm Q$}\hbox{\raise
0.15\ht0\hbox to0pt{\kern0.4\wd0\vrule height0.7\ht0\hss}\box0}}}}
\def\T{{\mathchoice {\setbox0=\hbox{$\displaystyle\rm
T$}\hbox{\hbox to0pt{\kern0.3\wd0\vrule height0.9\ht0\hss}\box0}}
{\setbox0=\hbox{$\textstyle\rm T$}\hbox{\hbox
to0pt{\kern0.3\wd0\vrule height0.9\ht0\hss}\box0}}
{\setbox0=\hbox{$\scriptstyle\rm T$}\hbox{\hbox
to0pt{\kern0.3\wd0\vrule height0.9\ht0\hss}\box0}}
{\setbox0=\hbox{$\scriptscriptstyle\rm T$}\hbox{\hbox
to0pt{\kern0.3\wd0\vrule height0.9\ht0\hss}\box0}}}}
\def\Z{{\mathchoice {\hbox{$\sf\textstyle Z\kern-0.4em Z$}}
{\hbox{$\sf\textstyle Z\kern-0.4em Z$}}
{\hbox{$\sf\scriptstyle Z\kern-0.3em Z$}}
{\hbox{$\sf\scriptscriptstyle Z\kern-0.2em Z$}}}}

\newcommand{\eqnoone}  
   {}
\newcommand{\eqnotwo}  
   {}

\newcounter{alf}

\newcommand{\adresa}[1]{\par\vspace*{-11pt}
                        \begin{flushright}
                        {\small
                        #1}
                        \end{flushright}
                        }

\makeatletter
\renewcommand*\l@section{\@dottedtocline{1}{1.5em}{2.3em}}
\makeatother

\usepackage[colorlinks,plainpages,urlcolor=blue,breaklinks]{hyperref}
\usepackage{tikz}
\usetikzlibrary{cd}
\usepackage{enumitem} 
\usepackage{mathtools}

\definecolor{dkgreen}{RGB}{0,100,0}

\newcommand{\N}{\mathbb{N}}
\renewcommand{\Z}{\mathbb{Z}}
\renewcommand{\Q}{\mathbb{Q}}
\newcommand{\R}{\mathbb{R}}
\renewcommand{\C}{\mathbb{C}}

\renewcommand{\P}{\mathbb{P}}
\newcommand{\CP}{\mathbb{CP}}

\renewcommand{\k}{\Bbbk}
\newcommand{\K}{\mathbb{K}}

\newcommand{\RR}{\mathcal{R}}

\newcommand{\A}{{\mathcal{A}}}
\newcommand{\BB}{{\mathcal{B}}}

\newcommand{\CC}{{\mathcal{C}}}

\newcommand{\XX}{\mathcal{X}}

\newcommand{\cE}{\mathcal{E}}
\newcommand{\cI}{\mathcal{I}}

\newcommand{\NN}{\mathcal{N}}
\newcommand{\cS}{\mathcal{S}}
\newcommand{\bR}{\boldsymbol{\RR}}

\newcommand{\h}{{\mathfrak{h}}}
\newcommand{\g}{{\mathfrak{g}}}
\newcommand{\B}{{\mathfrak{B}}}

\newcommand{\M}{\mathsf{M}}

\DeclareMathOperator{\rank}{rank}
\DeclareMathOperator{\rk}{rk}
\DeclareMathOperator{\corank}{corank}
\DeclareMathOperator{\gr}{gr}
\DeclareMathOperator{\im}{im}

\DeclareMathOperator{\ideal}{ideal}
\DeclareMathOperator{\id}{id}
\DeclareMathOperator{\ab}{{ab}}

\DeclareMathOperator{\Sym}{Sym}
\DeclareMathOperator{\ch}{char}
\DeclareMathOperator{\Hom}{{Hom}}

\DeclareMathOperator{\PGL}{{PGL}}

\DeclareMathOperator{\Ext}{{Ext}}
\DeclareMathOperator{\Hilb}{{Hilb}}

\DeclareMathOperator{\spn}{span}
\DeclareMathOperator{\ann}{{ann}}

\DeclareMathOperator{\Lie}{Lie}
\DeclareMathOperator{\Spec}{{Spec}}

\DeclareMathOperator{\loc}{loc}
\DeclareMathOperator{\irr}{irr}

\DeclareMathOperator{\supp}{{supp}}

\DeclareMathOperator{\ii}{i}

\newcommand{\OS}{A}
\newcommand{\qA}{\mathrm{q}A}

\newcommand{\cga}{\textsc{cga}}

\newcommand{\surj}{\twoheadrightarrow}
\newcommand{\inj}{\hookrightarrow}
\newcommand{\isom}{\xrightarrow{
   \,\smash{\raisebox{-0.4ex}{\ensuremath{\scriptstyle\cong}}}\,}}

\newcommand{\DS}{\displaystyle}
\newcommand{\abs}[1]{\left| #1 \right|}

\newcommand{\bwedge}{\mbox{\normalsize $\bigwedge$}}

\newcommand{\pcoor}[1]{%
  \begingroup\lccode`~=`: \lowercase{\endgroup
  \edef~}{\mathbin{\mathchar\the\mathcode`:}\nobreak}%
  [
  \begingroup
  \mathcode`:=\string"8000
  #1%
  \endgroup 
  ]
}

\newcommand{\arxiv}[1]
{\texttt{\href{http://arxiv.org/abs/#1}{arXiv:#1}}}

\numberwithin{equation}{section}

\newtheorem{theorem}{Theorem}[section]
\newtheorem{corollary}[theorem]{Corollary}
\newtheorem{lemma}[theorem]{Lemma}
\newtheorem{proposition}[theorem]{Proposition}

\newtheorem{example}[theorem]{Example}
\newtheorem{remark}[theorem]{Remark}
\newtheorem{conjecture}[theorem]{Conjecture}
\newtheorem{question}[theorem]{Question}
\newtheorem{problem}[theorem]{Problem}

\begin{document}

{\footnotesize%
\hfill
\begin{tabular}{c}
Bull. Math. Soc. Sci. Math. Roumanie\\
Tome 68 (116) No. 3, 2025, 349--382
\end{tabular}}

\vskip 1.2 true cm
\setcounter{page}{1}

\begin{center} {\bf Resonance varieties and Lie algebras of matroids and hyperplane arrangements} \\
          {by}\\
{\sc Alexandru Suciu$^{(1)}$}
\end{center}

\pagestyle{myheadings} 
\markboth{ Resonance and Lie algebras of matroids and arrangements }{A.~Suciu}

\centerline{{\em To  Bazil Br\^{i}nz\u{a}nescu on the occasion of his 80th birthday}}

\begin{abstract}
The resonance varieties, the holonomy Lie algebra, and the holonomy Chen Lie algebra 
associated with the  Orlik--Solomon algebra of a matroid provide an algebraic lens 
through which to examine the rich combinatorial structure of matroids and their 
geometric realizations as complex hyperplane arrangements. In this survey paper, 
we emphasize the commonalities between these objects but also the possible 
ways in which realizable and non-realizable matroids differ, leading to some 
open questions.
\end{abstract}

\begin{quotation}

\end{quotation}

\thispagestyle{empty}

\setcounter{tocdepth}{1} 
\tableofcontents

\section{Introduction}
\label{sect:intro}

Matroid theory provides a powerful framework for studying combinatorial structures 
arising from discrete sets, abstracting linear independence to bridge algebraic, 
geometric, and topological methods. A simple matroid $\M$ on a ground set 
$\cE$ of $n$ elements encodes dependencies among subsets, with its lattice 
of flats capturing combinatorial structure. When realized as a hyperplane 
arrangement $\A$ in $\C^\ell$, the matroid's lattice $L(\M)$ corresponds 
to the arrangement's intersection lattice $L(\A)$, linking combinatorics to 
the topology of the complement $M(\A) = \C^\ell \setminus \bigcup_{H \in \A} H$ 
and the group theory of the fundamental group $G = \pi_1(M(\A))$. This survey 
explores algebraic invariants---such as the Orlik--Solomon algebra, holonomy 
Lie algebra, and holonomy Chen Lie algebra---alongside algebro-geometric 
invariants like resonance varieties and Koszul modules, and combinatorial 
structures like multinets, examining their interplay and dependence on the 
matroid versus its realization.

The \emph{Orlik--Solomon algebra} $\OS(\M)$, defined as the quotient of the 
exterior algebra $\bwedge V$ on a free abelian group $V$ with basis $\{e_1, \dots, e_n\}$ 
by circuit relations, is a graded, graded-commutative algebra whose Betti numbers 
$b_q(\M) = \rank \OS^q(\M)$ reflect the matroid's dependencies. For realizable matroids, 
the OS-algebra $\OS(\M)$ is isomorphic to the cohomology ring $H^*(M(\A); \Z)$, 
providing a topological interpretation of these invariants \cite{OS}. 

The \emph{holonomy Lie algebra} $\h(\M)$, 
a quotient of the free Lie algebra $\Lie(V^{\vee})$ on the dual $\Z$-module $V^{\vee}$ 
by a quadratic ideal, encodes relations dual to the multiplication in $\OS(\M)$ \cite{Ko83, PS-imrn04}. 
Its graded ranks $\phi_r(\M) = \rank \gr_r(\h(\M))$ correspond to the lower central series 
ranks of $G$ in the realizable case, via an epimorphism $\h(\M) \to \gr(G)$ that is an 
isomorphism over $\Q$  \cite{PS-imrn04} but may introduce torsion in positive characteristic, as seen 
in the non-Fano arrangement (Example \ref{ex:non-fano}).

The \emph{holonomy Chen Lie algebra} $\h(\M)/\h(\M)''$, a streamlined quotient capturing 
the metabelian structure of $\h(\M)$, has graded ranks $\theta_r(\M) = \rank \gr_r(\h(\M)/\h(\M)'')$ 
that coincide with the Chen ranks of $G$ in the realizable case \cite{PS-imrn04}. 
These ranks, conjecturally expressed for large $r$ in terms of $k$-multinets on 
sub-matroids (Conjecture~\ref{conj:chen-ranks-mat}), provide a powerful algebraic 
framework for probing the combinatorial topology of matroids. 

The \emph{resonance varieties} $\RR^q_s(\M, \k)$, defined as the jump loci of the cohomology 
of the Aomoto complexes $(\OS(\M) \otimes \k, \delta_a)$ \cite{Fa97}, encode dependencies 
among circuits and flats. In degree $1$, the variety $\RR^1_s(\M, \C)$ decomposes 
into linear subspaces corresponding to $k$-multinets on sub-matroids, with essential 
components arising from orbifold maps $M(\A) \to \CP^1 \setminus \{k \text{ points}\}$ 
in the realizable case \cite{FY07, MB09}. 

The \emph{Koszul modules} $W_q(\M, \k)$ 
are finitely generated modules over the polynomial ring $S_{\k}=\Sym(V^{\vee} \otimes \k)$, 
defined as the homology modules of a free $S_{\k}$-chain complex associated to 
$\OS(\M) \otimes \k$ \cite{PS-mrl}. The supports of these modules are closely connected 
to the resonance varieties; in fact, in degree $1$, we have that $\supp(W_1(\M,\k))=
\RR^1_1(\M,\k)$ (away from $0$). Moreover, $W_1(\M,\k)$ is equal 
to the infinitesimal Alexander invariant $\B(\h(\M)) \otimes \k = \h'/\h''$ \cite{PS-imrn04}, 
with the Hilbert series of $W_1(\M, \k)$ equaling the generating series of $\theta_r(\M)$ 
(shifted by $2$) in characteristic $0$ (Proposition~\ref{prop:inf Massey}).

This interplay of invariants raises open questions, such as the linearity of higher-degree 
resonance varieties $\RR^q_s(\M, \C)$ for $q > 1$, which holds for arrangements but 
remains unresolved for non-realizable matroids (Question~\ref{quest:res-mat-lin}), 
potentially distinguishing their algebraic structure. Similarly, the presence of 
torsion in $\h(\M)$ or the structure of Koszul modules may differentiate realizable from 
non-realizable matroids (Problem~\ref{prob:high-koszul}). By reformulating the Chen ranks 
conjecture in terms of multinets, this survey aims to unify combinatorial and topological 
perspectives, addressing these questions and their implications for matroid theory.

The paper is organized as follows: Section~\ref{sect:matroids} introduces matroids 
and multinets. Section~\ref{sect:OS-holo} defines the Orlik--Solomon and holonomy 
Lie algebras. Section~\ref{sect:res} examines resonance varieties and their connections 
to Koszul modules and Chen ranks. Section~\ref{sect:hyp-arr} links these invariants to 
the topology of hyperplane arrangements, addressing known results and open questions.

\section{Matroids and multinets}
\label{sect:matroids}

We start with a brief introduction to matroids, which generalize linear independence 
in vector spaces, and multinets, which encode symmetric partitions of matroid elements. 
These combinatorial structures are fundamental to our study, as their properties, particularly 
the structure of the matroid's lattice of flats and the multinet's base locus, will be 
used to define the algebraic invariants of later sections. We'll also explore the connections 
between realizable multinets and Latin squares, which highlight a surprising link between 
geometry, combinatorics, and algebra.

\subsection{Matroids}
\label{subsec:mat}

First introduced by Whitney, the notion of a matroid can be axiomatized in many ways, 
providing a unifying framework for concepts across linear algebra, graph theory, 
discrete geometry, and the theory of hyperplane arrangements. As general background on 
matroid theory, we refer to the books by Oxley \cite{Oxley} and White \cite{White} 
and the survey by Wilson \cite{Wilson}. We mention here only the concepts that 
will be needed for our purposes.

A {\em matroid}\/ $\M$ is a pair $(\cE,\cI)$, where 
$\cE$ is a finite set called the {\em ground set}\/ 
(often chosen to be $[n]\coloneqq \{1,\dots,n\}$) and $\cI$
is a collection of subsets of $\cE$ that contains the empty set, 
is closed under taking subsets, and satisfies the following exchange axiom: 
If $S, T\in \cI$ and $\abs{S} > \abs{T}$, then there exists 
$u \in S \setminus T$ such that $T \cup \{u\} \in \cI$. 
The elements of $\cE$ are sometimes called points, while 
the elements of $\cI$ are called {\em independent sets}. 
A maximal independent set is called a {\em basis}, while a 
minimal dependent set is called a {\em circuit}. Clearly, a set 
is independent if and only if it contains no circuit as a subset. 
A matroid $\M$ is said to be {\em connected}\/ if for every pair 
of distinct elements $u,v\in \cE$, there is a circuit containing both of them.

The {\em rank}\/ of a subset $S\subset \cE$, denoted $\rk(S)$, is the size of 
the largest independent subset of $S$. In particular, $S$ is independent 
if and only $\rk(S)=\abs{S}$, and so the matroid can also be specified by 
its rank function. The rank of $\M$ is simply the rank of its ground set. 
A subset  $S\subset \cE$ is {\em closed}\/ if it is 
maximal for its rank; that is, $\rk(S\cup \{e\})>\rk(S)$ for all $e\in \cE\setminus S$. 
The closure $\overline{S}$ of a subset $S\subset \cE$ is the intersection of 
all closed sets containing~$S$. Closed sets are also called {\em flats}. 
We will consider only {\em simple}\/ matroids, defined by the condition 
that all subsets of size at most two are independent.

The set of flats of $\M$, ordered by inclusion, forms 
a lattice, $L(\M)$, whose atoms are the singletons in $\cE$. 
We will denote by $L_k(\M)$ the set of rank-$k$ flats, 
and by $L_{\le k}(\M)$  the sub-poset of flats of rank 
at most $k$. The join of two flats $X$ and $Y$ is given by 
$X\vee Y=\overline{X\cup Y}$, while the meet is given by 
$X\wedge Y=X\cap Y$, and the rank function satisfies the 
semimodular inequality: 
\begin{equation}
\label{eq:semimod}
\rk(X)+\rk(Y)\ge \rk(X\vee Y)+\rk(X\wedge Y)
\end{equation}
for all $X,Y\in L(\M)$. It follows that $L(\M)$ is a {\em geometric 
lattice}, i.e., a graded, semimodular lattice where every element is a 
join of atoms.  A flat $X$ is {\em irreducible}\/ if it cannot be expressed 
as the meet of other flats; we denote by $L^{\irr}(\M)$ the set 
of irreducible flats in $L(\M)$. A flat $X$ is {\em modular}\/ if 
the semimodular inequality \eqref{eq:semimod} holds as equality 
for any other flat $Y$. A matroid $\M$ is 
{\em supersolvable}\/ if $L(\M)$ contains a maximal chain 
$\emptyset \subsetneq X_1  \subsetneq  \cdots  \subsetneq 
X_{\rk(\M)-1} \subsetneq  X_{\rk(\M)}=\cE$ of modular elements.

The {\em M\"{o}bius function}\/ of the lattice of flats, 
$\mu\colon L(\M) \to \Z$, sends a flat $X$ to the integer 
$\mu(X)=\mu(\emptyset, X)$, defined inductively by $\mu(\emptyset)=1$ 
and $\mu(X)=-\sum_{Y\subsetneq X} \mu(Y)$.  
For instance, if $X\in L_2(\M)$, then $\mu(X)=\abs{X}-1$. 

Two operations on matroids are especially important for us. 
First, the {\em localization}\/ of a matroid $\M$ at a flat $X\in L(\M)$ 
is the matroid $\M_X$ whose ground set is $X$ and whose independent 
sets are the independent sets of $\M$ completely contained in $X$. 
(This operation is also known as the restriction of $\M$ to $X$.) 
And second, given two matroids, $\M_1=(\cE_1,\cI_1)$ and 
$\M_2=(\cE_2,\cI_2)$ with disjoint ground sets, 
their {\em direct sum}, $\M_1\oplus \M_2$, is the matroid with 
ground set $\cE_1 \sqcup \cE_2$ and collection 
of independent sets $\{ I_1 \sqcup I_2 : I_1 \in \cI_1, I_2 \in \cI_2 \}$. 

An {\em arrangement of hyperplanes}\/ is a finite set $\A$ of 
codimension-$1$ linear subspaces in a finite-dimensional, 
complex vector space $\C^{\ell}$.  
The combinatorics of the arrangement is encoded in its 
intersection lattice, $L(\A)$; this is the poset of all 
intersections of $\A$, ordered by reverse inclusion, and 
ranked by codimension (see \S\ref{subsec:arr-lat} for more on this). 
A matroid $\M$ is said to be {\em realizable}\/ (over $\C$) 
if there is an arrangement $\A$ such that $L(\M)=L(\A)$. 

\begin{example}
\label{ex:unif-mat}
The {\em uniform matroid}\/ $U_{r,n}$ of rank $r\ge 2$ on ground set $[n]$ with $n\ge 2$ 
is characterized by the property that all bases have size $r$. This is a simple, 
connected matroid, with rank function $\rk(S) = \min(\abs{S}, r)$. Its poset of flats 
consists of all subsets of $[n]$ with of size less than $r$ or equal to $n$. 
It is supersolvable if and only if $r=2$ or $r=n$, and it can be realized by a generic 
arrangement of $n$ hyperplanes in $\C^r$. 
A special case is $U_{n,n}$, the {\em free matroid}\/ of rank $n$, for which 
all subsets of $[n]$ are independent and the lattice of flats is the Boolean lattice. 
This matroid can be realized by the 
Boolean arrangement of coordinate hyperplanes in $\C^n$. 
\end{example}

\begin{example}
\label{ex:graphic-mat}
To a simple graph $\Gamma$ (a graph with no loops or multiple edges) 
on vertex set $[n]$ there corresponds the {\em graphic 
matroid}\/ $\M(\Gamma)=\{\cE, \cI)$ on ground set the edge set of $\Gamma$  
and with independent sets the forests in $\Gamma$. Such a matroid 
can be realized by the arrangement $\A(\Gamma)$ in $\C^n$ consisting of the 
hyperplanes $H_{i,j}=\{z_i-z_j=0\}$ for all edges $\{i,j\}\in \cE$. In particular, 
if $\Gamma$ is the complete graph on $n$ vertices $K_n$, then $\A(K_n)$ is 
the reflection arrangement of type $\operatorname{A}_{n-1}$, also known as the 
braid arrangement in $\C^n$. As shown by Stanley \cite{Stanley}, a graphic 
matroid $\M(\Gamma)$ is supersolvable if and only if the graph $\Gamma$ 
is chordal. 
\end{example}

\subsection{Multinets on matroids}
\label{subsec:multinets}

Multinets on matroids are combinatorial structures that arise from a partition 
of a matroid's ground set. These objects, introduced by Falk and Yuzvinsky \cite{FY07} 
and Marco Buzun\'{a}riz \cite{MB09}, play a fundamental role in the study of hyperplane 
arrangements and their resonance varieties.

A {\em weak $(k,d)$-multinet}\/ on a matroid $\M = (\cE, \cI)$ is a partition of the 
ground set $\cE = \cE_1 \sqcup \cdots \sqcup \cE_k$ into $k \ge 3$ non-empty 
subsets, called parts, together with a multiplicity function $m \colon \cE \to \N$, 
such that the following conditions are met:
\begin{enumerate}[label=(\arabic*), itemsep=2pt]
\item \label{mu1}
The sum of multiplicities in each part is a constant integer $d \ge 1$:
\[
\sum_{u \in \cE_\alpha} m_u = d \quad \text{for all } \alpha \in [k].
\]
\item \label{mu2}
For any two points $u,v \in \cE$ in distinct parts, their span $u \vee v$ is 
an irreducible flat of rank $2$. The set
\[
\XX = \{ u \vee v : u \in \cE_{\alpha}, v \in \cE_{\beta} \text{ for } \alpha \ne \beta \}
\]
of such flats is called the {\em base locus}\/ of the multinet.

\item \label{mu3}
For each rank-$2$ flat $X$ in the base locus $\XX$, the sum of multiplicities 
of points in $X$ belonging to any single part is a constant integer $n_X$ 
(called the order of the flat):
\[
n_X = \sum_{w \in X \cap \cE_\alpha} m_w \quad \text{is independent of } 
\alpha \in [k] \text{ such that } X \cap \cE_\alpha \neq \emptyset.
\]
\end{enumerate}

A weak multinet, denoted by $\NN$, is called a {\em multinet}\/ if it satisfies 
the following additional connectivity condition:
\begin{enumerate}[label=(\arabic*), start=4, itemsep=2pt]
\item \label{mu4}
For any two points $u,v$ within the same part $\cE_\alpha$, there exists a 
path of points in $\cE_\alpha$ connecting them, where each consecutive 
pair of points spans a flat that is {\em not}\/ in the base locus $\XX$.
\end{enumerate}

We sometimes refer to these objects as $(k,d)$-multinets or simply $k$-multinets.
Without an essential loss of generality, we may assume that
$\gcd\{m_u\}_{u\in \cE}=1$. It follows from the definitions (see \cite{FY07, Su-toul})  
that the following equalities must hold:
\begin{enumerate}[label=(\roman*), itemsep=2pt]
\item The sum of all multiplicities equals the product of the number of 
classes and the degree: $\sum_{u\in \cE} m_u=k d$. 
\item For any given element in the ground set, the sum of the orders 
over all base locus flats containing it equals the degree: 
$\sum_{X\in \XX:  u\in  X} n_X= d$, for any $u\in \cE$. 
\item The sum of squares of the orders of all base locus flats is equal 
to the square of the degree: $\sum_{X\in \XX} n_X^2 = d^2$.
\end{enumerate}

When the matroid $\M$ is realizable by a hyperplane arrangement, 
there is an additional relationship 
between the numerical data associated to $\M$ and a multinet on it. 
Indeed, as shown in \cite[Theorem~4.2]{FY07}, the following 
Riemann--Hurwitz type formula holds in that case:
\begin{equation}
\label{eq:RH}
3+\abs{\XX} \ge 2\abs{\cE}- (k-2) \Big( 3d - \sum_{X\in \XX} n_X \Big) 
- \sum_{X\in L_2(\M)\setminus \XX} \mu(X) .
\end{equation}

\begin{question}
\label{quest:EH}
Is the inequality \eqref{eq:RH} satisfied for $(k,d)$-multinets on 
non-realizable matroids?
\end{question}

As noted in \cite[Remark 3.6]{FY07}, every weak multinet can be
refined to a multinet with the same base locus $\XX$; the refinement 
breaks down some of the original parts into smaller, single-element parts, 
which then satisfy the connectivity condition. Here is a 
concrete example that illustrates how this may be done.

\begin{example}
\label{ex:weak-multinet}
Let $\M=U_{2,3}\oplus U_{2,3}$ be the matroid on ground set $[6]$ 
whose set of circuits is $\CC=\{\{1,2,3\}, \{4,5,6\}\}$. Then $\M$ admits a weak $(3,2)$-multinet 
$\NN$ with parts $\{1,4\}, \{2,5\}, \{3,6\}$, all multiplicities equal to $1$, and base locus 
$\XX=\CC=L_2(\M)$. Observe that $\NN$ is not a multinet, as condition \ref{mu4}
is not satisfied: points $1$ and $4$ are in the same part, but cannot be connected 
by a chain of $2$-flats outside $\XX$, as all $2$-flats are circuits.
Nonetheless, $\NN$ may be refined to a multinet $\NN'$ whose parts are 
the singletons, and whose multiplicities and base locus are the same as those of $\NN$.
\end{example}

\subsection{Reduced multi-nets and nets}
\label{subsec:red multi}

Once again, let $\NN$ be a $(k,d)$-multinet on a matroid $\M=(\cE,\cI)$. 
If all the multiplicities $m_{u}$ for $u\in \cE$ are equal to $1$, the multinet 
is said to be {\em reduced}.  If $n_X=1$, for all $X\in \XX$, 
the multinet is called a {\em $(k,d)$-net};  in this case, 
the multinet is reduced, and every flat in the base locus contains precisely 
one element from each class. Moreover, $\abs{\cE}=kd$ and $\abs{\XX}=d^2$. 
A $(k,d)$-net $\NN$ is {\em non-trivial}\/ if $d>1$, 
or, equivalently, if the matroid $\M$ has rank at least $3$.
If $\M$ has no $2$-flats of size larger than $3$, then any
multinet on $\M$ is a $3$-net, see \cite{DF17}. 

Let $\NN$ be a multinet on a matroid $\M=(\cE,\cI)$, with parts 
$\cE_1,\dots, \cE_k$. For each flat $X\in L_2(\M)$, let us write 
$\supp_{\NN}(X)=\{\alpha \in [k] : \cE_{\alpha}\cap X \ne \emptyset\}$. 
Evidently, $|\!\supp_{\NN}(X)| \le \abs{X}$.  
Notice also that $|\!\supp_{\NN}(X)|$ is either equal to $1$ (in which case 
we say $X$ is {\em mono-colored}), or to $k$ (in which case we say 
$X$ is {\em multi-colored}).  As shown in \cite{PS-plms17}, 
if $\M$ has no $2$-flats of multiplicity $kr$, for any $r>1$, 
then every reduced $k$-multinet on $\M$ is a $k$-net. 

Work of Kawahara \cite{Ka07} shows that nets on matroids abound:    
for any $k\ge 3$, there is a (simple) matroid $\M$ supporting a non-trivial $k$-net. 
On the other hand, work of Yuzvinsky \cite{Yu04, Yu09} 
and Pereira--Yuzvinsky \cite{PY08} 
shows that, if $\NN$ is a $k$-multinet on a realizable matroid, 
with base locus of size greater than $1$, then $k=3$ 
or $4$; furthermore, if $\NN$ is not reduced, then $k$ 
must equal $3$.

Given a matroid $\M$ and  a subset $S$ of its ground set $\cE$, 
we say that $S$ is {\em line-closed}\/ 
(in $\M$) if $X$ is closed in $\M$, for all $X\in L_2(S)$.  
Clearly, this property is stable under intersection.
As shown in \cite{PS-plms17}, if $\M$ supports a $k$-net, then 
each part $\cE_{\alpha}$ is line-closed in $\M$. 
An alternative characterization of nets was given in \cite{PS-plms17}: 
A $k$-net on a matroid $\M$ is a partition of its ground set into 
non-empty blocks, $\cE =\cE_{1}\sqcup \cdots \sqcup \cE_{k}$, with the property that, 
for every $u\in \cE_{\alpha}$ and $v\in \cE_{\beta}$ with $\alpha \ne \beta$ 
and every $\gamma\in [k]$, 
\begin{equation}
\label{eq=altnet}
\abs{(u \vee v) \cap \cE_{\gamma}}=1.
\end{equation}
Consequently, a $3$-net on a matroid $\M=(\cE,\cI)$ is simply a partition of $\cE$ 
into three non-empty subsets $\cE_1, \cE_2, \cE_3$  with the property that, for each 
pair of points $u,v\in \cE$ in different parts, we have $u\vee v=\{ u,v,w \}$, for some 
point $w$ in the third part. 

\subsection{Latin squares and realizable $3$-nets}
\label{subsec:latin}

The last observation points to an intimate connection between $3$-nets and a 
combinatorial structure that goes back to antiquity. A {\em Latin square}\/ of size 
$d$ is a $d \times d$ matrix where each row and each column is a permutation 
of the set $[d]$; as such, it represents the multiplication table for a quasi-group. 
Now, if a matroid $\M$ on ground set $\cE$ admits a $(3,d)$-net 
with parts $\cE_1, \cE_2, \cE_3$, then the multi-colored rank-$2$ flats define a 
Latin square $\Lambda$. To see this, we label the points in the first two parts as 
$\{u_1^1, \dots, u_d^1\}$ and $\{u_1^2, \dots, u_d^2\}$. The net's properties 
ensure that for each pair of points $(u_p^1, u_q^2)$, there is a unique point in 
the third part, $u_r^3 \in \mathcal{E}_3$, such that $\{u_p^1, u_q^2, u_r^3\}$ 
forms a rank-$2$ flat. We then construct the matrix $\Lambda$ where the 
$(p,q)$-entry is the integer $r$. A similar procedure shows that a $k$-net 
is encoded by a $(k-2)$-tuple of orthogonal Latin squares.

Particularly simple is the following construction, due to 
Kawahara \cite{Ka07}: given any Latin square, there is a 
matroid with a $3$-net realizing it, such that each submatroid 
obtained by restricting to the parts of the $3$-net is a uniform matroid. 
In turn, some of these matroids may be realized by 
line arrangements in $\CP^2$. 

Given a finite group $H$ of order $d$, we say that $H$ may be 
realized by a $3$-net in $\CP^2$ if there is a $(3, d)$-net on a 
line arrangement in $\CP^2$ and there is a way to 
identify the blocks of the net to obtain the
multiplication table of $H$. For instance, if $H$ is equal 
to $\Z_2$, $\Z_3$, $\Z_4$, or $\Z_2\times \Z_2$, 
then the corresponding realizations are the braid arrangement, 
the Pappus $(9_3)_1$ configuration, the Kirkman configuration, 
and the Steiner configuration, respectively. 
In \cite{Yu04, Yu12}, Yuzvinsky gave several conditions under 
which a finite group $H$ may be realized by a $3$-net in $\CP^2$. 
For example, if $H$ is a subgroup of the torus $S^1\times S^1$, then 
$H$ can be realized, but the group $\Z_2^3$ cannot. 
Moreover, if $H$ is an abelian group that has an element of order 
at least $10$, then $H$ can be realized by a $3$-net in $\CP^2$ if and only
if $H$ has at most two invariant factors. Pereira and Yuzvinsky \cite{PY08} 
realized all the dihedral groups, while 
Urz\'ua \cite{Uz} realized the quaternion group $Q_8$. 
Work of Korchm\'aros, Nagy, and Pace \cite{NP13, KNP14} 
shows that these are the only other finite subgroups 
of $\PGL(2,\C)$ that can be realized by a $3$-net in $\CP^2$, 
thereby answering a question from \cite{Yu12}.

A $3$-net in $\CP^2$ is said to be {\em algebraic}\/ if there is a cubic curve $C$ 
in the dual projective plane (the space whose points correspond to the lines 
of the original plane) such that $C$ contains the points dual to the lines of the net 
in the set of its regular points. As shown in \cite{Yu12}, if $H$ is a finite 
abelian group that is either cyclic or has at least one element of order 
at least $10$, then every realization of $H$ by a $3$-net in $\CP^2$ is algebraic.
Moreover, as shown in \cite{Yu12, KNP14}, every $3$-net realizing a finite 
cyclic group or the direct product of two cyclic groups is algebraic. We 
end this section with a question posed by Yuzvinsky in \cite{Yu12}.

\begin{question}[\cite{Yu12}]
\label{quest:alg-net}
Are all $3$-nets in $\CP^2$ algebraic? 
\end{question}

With these core combinatorial structures---matroids, their lattices of flats, and multinets---in place, 
we have established the necessary foundation. The discussion of realizable nets and their 
connection to Latin squares and finite groups demonstrates that purely combinatorial properties 
of matroids can have profound implications for their geometric and algebraic realizations. 
In the following section, we will build upon this foundation by introducing new algebraic 
tools needed to probe these relationships from a different viewpoint.

\section{Algebras associated to matroids}
\label{sect:OS-holo}

This section introduces two algebraic structures associated with a matroid $\M$, 
each capturing its combinatorics in a distinct manner. The Orlik--Solomon algebra 
$\OS(\M)$, a graded anti-commutative algebra, encodes the matroid’s circuits 
and their dependencies. The holonomy Lie algebra $\h(\M)$, defined via relations among 
the $2$-flats, encodes the lattice of flats’ dependencies in a Lie-theoretic framework. 
Together with related Lie and homological invariants (the holonomy Chen Lie algebra 
and the infinitesimal Alexander invariant), these algebras illuminate matroid combinatorics 
and their topological connections to hyperplane arrangements.

\subsection{The Orlik--Solomon algebra}
\label{subsec:OS}
Given a simple matroid $\M$, let $V=\bigoplus_{u\in\cE} \Z e_u$ be the free 
abelian group on the ground set $\cE$, with its distinguished basis $\{e_u\}_{u\in\cE}$. 
Furthermore, let $E=\bwedge V$ be the exterior algebra on $V$, with graded pieces 
$E^k=\bwedge^k V$ for $0\le k\le n$, where $n=\abs{\cE}$. 
Define a degree $-1$ derivation $\partial \colon E \to E$ by setting $\partial(1) = 0$, 
$\partial(e_u) = 1$, and then extending to the whole of $E$ via the graded Leibniz rule, 
\begin{equation}
\label{eq:partial}
    \partial(e_{u_1}\wedge \dots \wedge e_{u_k}) = \sum_{i=1}^k (-1)^{i-1} e_{u_1} 
    \wedge \dots \wedge \widehat{e_{u_i}} \wedge \dots \wedge e_{u_k} 
 \end{equation}
for $k\ge 2$.
For a $k$-tuple $S=(u_1,\dots, u_k)$ of elements in $\cE$, we denote by $e_S$  
the element $e_{u_1} \wedge \dots \wedge e_{u_k}$ of $E^k$. Let $I=I(\M)$ 
be the ideal of $E$ generated by the elements $\partial(e_S)$, with 
$S$ a circuit of $\M$. Since all elements $\partial(e_S)$ are homogeneous, 
$I=\bigoplus_{k\ge 2} I^k$ is a graded ideal, generated in degrees 
$2$ and higher. The (integral) \emph{Orlik--Solomon algebra}\/ of 
the matroid $\M$ is defined as 
\begin{equation}
\label{eq:OS}
\OS(\M)=E/I(\M).
\end{equation}
This algebra was first defined by Orlik and Solomon in \cite{OS} when 
the matroid $\M$ can be realized by a complex hyperplane arrangement $\A$, in 
which case $\OS(\M)$ is isomorphic to the cohomology ring of the complement 
of $\A$ (see \S\ref{subsec:arr-lat}).

The OS-algebra $A=\OS(\M)$ is a graded, graded-commutative, $\Z$-algebra, 
with graded pieces $A^k=E^k/I^k$. Since $I^0=I^1=0$, the canonical projection 
$\pi_{\M}\colon E \to A$ identifies $A^0=E^0 \cong \Z$ and 
$A^1=E^1\cong \Z^n$. By a slight abuse of notation, we will also 
denote by $e_u$ the basis elements $\pi_{\M}(e_u)$ of $A^1$. 
It turns out that each of the graded pieces $A^k$ is a free 
abelian group, of rank 
\begin{equation}
\label{eq:betti}
b_k(\M)=(-1)^k \sum_{X\in L_k(\M)}  \mu(X),
\end{equation}
where recall $\mu\colon L(\M) \to \Z$ 
is the M\"{o}bius function of the lattice of  
flats of $\M$. We will denote by 
\begin{equation}
\label{eq:poin}
P(\M,t)\coloneqq \Hilb(A(\M),t)=\sum_{k\ge 0} b_k(\M) t^k
\end{equation}
the Poincar\'e polynomial of $\M$. Plainly, the OS-algebra $\OS(\M)$ only depends 
on the lattice of flats $L(\M)$. 

When the matroid decomposes as a direct sum, 
$\M=\M_1\oplus \M_2$, its lattice of flats decomposes as a direct product, 
$L(\M)=L(\M_1)\times L(\M_2)$, from which it follows that its OS-algebra 
decomposes as a tensor product, $\OS(\M)\cong \OS(\M_1)\otimes \OS(\M_2)$.

Since $\partial^2=0$ and the ideal $I$ is generated by elements of the form 
$\partial(e_S)$, we have that $\partial(I)=0$. Therefore, the map $\partial\colon E\to E$ 
induces a derivation $\partial_A\colon A\to A$.  Consider the homomorphism 
$\left.\partial\right|_{V}\colon V\to \Z$ which sends each basis element 
$e_u$ to $1$. Let 
\begin{equation}
\label{eq:vbar}
\overline{V} = \ker(\partial|_{V})=
\Big\{w\in V \colon \sum_{u\in \cE}  w_u=0\Big\} .
\end{equation}
As observed in \cite{De16}, the subalgebra $\overline{A}$ of $A$ 
generated by $\overline{V}$ coincides with $\ker(\partial_A)$.
The subalgebra $\overline{A}=\overline{A}(\M)$ is called the 
{\em projective Orlik--Solomon algebra}\/ of $\M$; the terminology 
comes from the geometric interpretation of $\overline{A}$ in the 
realizable case as the cohomology ring 
of the projectivized arrangement complement (see \S\ref{subsec:arr-lat}).

An Orlik--Solomon algebra $A=E/I$ is called {\em quadratic}\/ if 
the OS-ideal $I$ is generated in degree $2$, that is, $I=(I^2)$. 
A much stronger condition is that $I=I(\M)$ admit a quadratic Gr\"obner basis; 
as shown by Bj\"orner and Ziegler \cite{BZ}, this happens precisely when 
$\M$ is supersolvable.

\begin{example}
\label{ex:os-uniform}
Let $\M=U_{r,n}$ be the uniform matroid of rank $r>1$ on ground set $[n]$. 
The circuits in $L(\M)$ are all the subsets of $[n]$ of size $r+1$. Thus, 
the OS-algebra $A=A(\M)$ is the quotient of the exterior algebra 
$E$ on generators $e_1,\dots, e_n$ by the ideal $I$ generated 
by $\partial (e_S)$ for all $S\in\binom{[n]}{r+1}$. Therefore, 
$A$ is quadratic if and only if $\M$ is the free matroid $U_{n,n}$, 
in which case $A=E$ and $\M$ is supersolvable.
\end{example}

\begin{example}
\label{ex:super}
Let $\M=\M(K_n)$ be the graphic matroid of the complete graph $K_n$, 
whose lattice of flats is the lattice of partitions of the set $[n]$ (with 
opposite order). 
As noted in Example \ref{ex:graphic-mat}, this matroid is supersolvable. 
Moreover, the OS-algebra of $\M$ is the quotient of the exterior algebra 
$E$ on generators $\{ e_{ij} : 1\le i < j \le n \}$ by the ideal $I$ generated 
by the set $\{ e_{ij}e_{jk} - e_{ij}e_{ik} + e_{jk}e_{ik} : 1\le i < j < k\le n \}$, 
which forms a quadratic Gr\"obner basis for $I$. Finally, the Poincar\'e polynomial 
of $\M$ factors as $P(\M,t)=\prod_{j=1}^{n-1} (1+jt)$.
\end{example}

\subsection{Holonomy Lie algebra}
\label{subsec:holo}

We now describe another algebraic invariant that captures the combinatorics 
of the matroid $\M$, but from a different perspective (a dual one). 
This is the holonomy Lie algebra $\h(\M)$, which was first considered 
by Kohno \cite{Ko83} in the context of hyperplane arrangements, and then in the 
works of several authors, including 
\cite{SY97, PS-imrn04, PS-cmh06, DeS06, LS09, De10, 
Lofwall-16, Lofwall-20, PS-ejm20, Su-decomp}. 

As before, let $V$ be the free abelian group on the ground set $\cE$. 
The dual $V^{\vee}=\Hom(V,\Z)$ of the free abelian group $V$ is also 
free abelian, with dual basis $\{x_u=e_u^{\vee}\}_{u\in \cE}$. 
Let $\Lie (V^{\vee})$ be the free Lie algebra on $V^{\vee}$. 
This is a graded Lie algebra, with grading given by bracket-length, 
and it is generated by its degree $1$ piece, $\Lie_1(V^{\vee})=V^{\vee}$. 
Since $V$ has finite rank, the $\Z$-linear map $[x_u,x_v]\mapsto x_u\wedge x_v$ 
identifies $\Lie_2(V^{\vee})$ with $V^{\vee}\wedge V^{\vee}=E^2$. Each graded 
piece $\Lie_r(V^{\vee})$ is a free abelian group, of rank 
\begin{equation}
\label{eq:witt}
\frac{1}{r}\sum_{d\mid r} \mu(d) n^{r/d},
\end{equation}
where $n=\rank V=\abs{\cE}$ 
and $\mu\colon \N\to \Z$ is the classical M\"{o}bius function; in particular, 
if $r$ is a prime, then $\Lie_r(V^{\vee})$ has rank $(n^r-n)/r$.

Let $\pi_{\M}^\vee\colon \OS^2(\M)^{\vee}\to V^{\vee}\wedge V^{\vee}$ 
be the map dual to the canonical projection $\pi_{\M}\colon E^2 \to \OS^2(\M)$. 
The {\em holonomy Lie algebra}\/ of $M$ is defined as
\begin{equation}
\label{eq:holo-Lie}
    \h(\M)=\Lie (V^{\vee}) / \ideal \big(\! \im \big( \pi_{\M}^{\vee} \big) \big),
\end{equation}
the quotient of the free Lie algebra on $V^{\vee}$ by the Lie ideal generated by the 
image of $\pi_{\M}^{\vee}$, viewed as a subset of $\Lie_2(V^{\vee})$ via the 
aforementioned identifications. Since this ideal is a homogeneous (quadratic) ideal, 
the holonomy Lie algebra $\h=\h(\M)$ inherits the structure of a graded Lie algebra from 
that of $\Lie (V^{\vee})$. We shall denote its graded pieces by $\h_r(\M)$ for $r\ge 1$, 
and their ranks by 
\begin{equation}
\label{eq:lcs-ranks}
\phi_r(\M) = \rank \h_r(\M).
\end{equation}
Both $\h_1$ and $\h_2$ are free abelian groups, of ranks 
$\phi_1= n$ and $\phi_2=\sum_{X\in L_2(\M)} \binom{\mu(X)-1}{2}$, respectively. 

By construction, the holonomy Lie algebra of $\M$ is determined by $L_{\le 2}(\M)$, 
the lattice of flats of $\M$ truncated in degree $2$. 
An explicit presentation for this Lie algebra (first given by Kohno \cite{Ko83} in 
the realizable case) is as follows:
\begin{equation}
    \label{eq:holo-pres}
    \h(\M)=\Lie( x_u: u\in \cE)\Big\slash \ideal \,
    \Big\{ \, \Big[x_u , \sum_{v\in X} x_{v}\Big] : 
    X\in L_2(\M), \ 
    u\in X \,
    \Big\}  .
\end{equation}
The relations in $\h(\M)$ can be viewed as enforcing compatibility with the 
circuits of rank-$2$ flats. As noted previously, if $\M=\M_1\oplus \M_2$, we 
have that $L(\M)=L(\M_1)\times L(\M_2)$ and 
$\OS(\M)\cong \OS(\M_1)\otimes \OS(\M_2)$; 
either of these decompositions implies that 
$\h(\M)\cong \h(\M_1)\times \h(\M_2)$.

\begin{example}
\label{ex:holo-U2n}
Let $\M=U_{2,n}$ be the uniform matroid of rank $2$ on ground set $[n]$. 
The holonomy Lie algebra of $\M$ is the quotient 
of $\Lie(x_1,\dots ,x_n)$ by the ideal generated by $[x_i,\sum_{j=1}^n x_j]$ 
for $1\le i\le n$. Hence, $\h(\M)$ is isomorphic to $\Lie(\Z^{n-1})\times \Lie(\Z)$, 
and thus its graded ranks are given by $\phi_1=n$ and 
$\phi_r=\tfrac{1}{r}\sum_{d\mid r} \mu(d) (n-1)^{r/d}$ for $r\ge 2$. 
In particular, when $n=2$, we have that $\h(\M)\cong \Lie(\Z^2)$ 
and $\phi_r=0$ for $r\ge 2$. 
\end{example}

\begin{example}
\label{ex:holonomy-graph}
Let $\M=\M(\Gamma)$ be the graphic matroid associated to a simple graph 
$\Gamma$ with edge set $\cE$.  Then $\h(\M)$ is the quotient of 
$\Lie(\cE)$ by the ideal generated by the brackets of the form 
$[e_1 , e_2 + e_3]$ for all triples of edges $\{e_1,e_2,e_3\}$ 
forming a triangle in $\Gamma$ and by the brackets 
$[e_1, e_2]$ for all pairs of edges $\{e_1,e_2\}$
for which $\{e_1,e_2,e\}$ is not a triangle in $\Gamma$, 
for any $e\in \cE$. We will come back to the holonomy ranks 
of $\h(\M)$ in Example \ref{ex:lcs-graph}.
\end{example}

\subsection{Computing the holonomy ranks}
\label{subsec:phik}

Fix a field $\k$ of characteristic $0$ and let $\h=\h(\M)\otimes \k$ 
be the holonomy Lie algebra of the matroid $\M$ over $\k$. Its 
universal enveloping algebra, $U(\h)$, is a quadratic algebra, endowed 
with the bracket-length filtration induced from that of $\h$. 
By the Poincar\'{e}--Birkhoff--Witt theorem, 
$\gr(U(\h))\cong \Sym(\h)$, and therefore 
\begin{equation}
\label{eq:pbw}
\Hilb(U(\h),t)=\prod_{k\ge 1}(1-t^r)^{-\phi_r},
\end{equation}
where $\phi_r=\phi_r(\M)=\dim_{\k} \h_r$. 
Work of Shelton--Yuzvinsky \cite{SY97} 
(see also \cite{PaY, SW-forum} for related results) shows that 
\begin{equation}
\label{eq:sy}
U(\h)=(\qA)^!,
\end{equation}
where $\qA=E/(I^2) \otimes \k$ is the quadratic closure of $A_{\k}=A\otimes \k$ 
and $(\qA)^!$ is its quadratic dual (the defining relations of $(\qA)^!$ form the orthogonal 
complement to those of $\qA$ inside $A^1_{\k} \otimes_{\k} A^1_{\k}$, see \cite{PP}).
Computing the holonomy ranks 
$\phi_r$, then, is equivalent to computing the Hilbert
series of the quadratic dual of the Orlik--Solomon algebra of $\M$ over $\k$. 

On the other hand, L\"{o}fwall showed in \cite{Lofwall-86} that the quadratic 
dual of $A_{\k}$ is isomorphic 
to $[\Ext^1_{A_{\k}}(\k,\k)]\coloneqq \bigoplus_{i\ge 0} \Ext^i_{A_{\k}}(\k,\k)_i$, 
the linear strand in the Yoneda algebra $\Ext_{A_{\k}}(\k,\k)$. It then follows from 
\eqref{eq:pbw} and \eqref{eq:sy} that 
\begin{equation}
\label{eq:lin-strand}
\prod_{r\ge 1}(1-t^r)^{\phi_r} = \sum_{i\ge 0} b_{ii} t^i , 
\end{equation}
where $b_{ij}=\dim_{\k} \Ext^i_{A_{\k}}(\k,\k)_j$ are the bigraded Betti numbers 
of the Ext-algebra of $A_{\k}$. Therefore, computing the holonomy ranks 
$\phi_r$ is also equivalent to computing the Betti numbers of the linear strand 
of a minimal free resolution of $\k$ over $A_{\k}$.

Now suppose $\M$ is supersolvable.  Then, as noted in \S\ref{subsec:OS}, 
the OS-algebra $A_{\k}$ admits a quadratic Gr\"obner basis. Therefore, 
$A_{\k}$ is a {\em Koszul algebra}, that is, $b_{ij}=0$ if $i\ne j$, and hence, 
the Hilbert series of $A_{\k}$ and its quadratic dual are related by the 
Koszul duality formula,
\begin{equation}
\label{eq:Koszul-dual}
\Hilb(A_{\k},t) \cdot \Hilb(A_{\k}^!,-t)=1, 
\end{equation}
see for instance \cite{PP}. Using now formula \eqref{eq:pbw}, 
we conclude that the holonomy ranks 
$\phi_k=\phi_k(\M)$ are determined by the Betti numbers 
$b_k=b_k(\M)$ from \eqref{eq:betti} by the following equality in $\Z[\![t]\!]$:
\begin{equation}
\label{eq:lcs-ss}
\prod_{r\ge 1}(1-t^r)^{\phi_r} =\sum_{k\ge 0} (-1)^k b_kt^k .
\end{equation}

\begin{example}
\label{ex:lcs-kn}
As noted in Example \ref{ex:super}, the graphic matroid $\M(K_n)$ associated to the 
complete graph on $n$ vertices is supersolvable. Applying \eqref{eq:lcs-ss}, one finds that 
the holonomy ranks $\phi_r=\phi_r(\M(K_n))$ are given by 
$\prod_{r\ge 1}(1-t^r)^{\phi_r} =\prod_{j=1}^{n-1} (1-jt)$, a formula originally 
proved by Kohno \cite{Ko85} and Falk and Randell \cite{FR85} by different methods. 
\end{example}

\begin{example}
\label{ex:lcs-graph}
More generally, let $\M(\Gamma)$ be the graphic matroid associated to a graph 
$\Gamma$ on $n$ vertices, and let $\kappa_s$ be the number of complete subgraphs 
on $s$ vertices. The following formula, conjectured by Schenck and Suciu in \cite{SS-tams02} 
and proved by Lima-Filho and Schenck in \cite{LS09}, computes the holonomy ranks 
$\phi_r=\phi_r(\M(\Gamma))$:
\begin{equation}
\label{eq:GLCS}
\prod_{r=1}^{\infty} \left(1-t^r\right)^{\phi_r}=
\prod_{j=1}^{n-1} \big(1-j t\big)^{\DS{\sum_{s=j}^{n-1}(-1)^{s-j}
\tbinom{s}{j} \kappa_{s+1}
}} .
\end{equation}
\end{example}

Examples of (finitely presented) graded algebras which are Koszul but 
do not admit a quadratic Gr\"obner basis do exist, see e.g.~\cite{PP}. 
Yet, as noted by Yuzvinsky in \cite{Yu01}, no such example is known 
among OS-algebras of hyperplane arrangements. We then have this 
question, which is still open, some 25 years later (for either realizable 
of non-realizable matroids).

\begin{question}[\cite{Yu01}]
\label{quest:koszul}
Suppose the Orlik--Solomon algebra of a matroid $\M$ (over a field $\k$) 
is a Koszul algebra. Does it follow that the matroid is supersolvable?
\end{question}

\subsection{Holonomy Chen Lie algebra}
\label{subsec:holo-Chen}
The derived series of a Lie algebra $\g$ is defined recursively by $\g^{(0)}=\g$ 
and $\g^{(i+1)}=[\g^{(i)},\g^{(i)}]$ for $i\ge 1$. The terms $\g'=\g^{(1)}$ and $\g''=\g^{(2)}$ 
are known as the derived Lie algebra and the second derived Lie algebra of $\g$, 
respectively. The quotients $\g/\g^{(i)}$ are the maximal 
solvable quotients of $\g$; in particular, $\g/\g'$ is the abelianization  
and $\g/\g''$ is the maximal metabelian quotient of $\g$. 

Now suppose $\g$ is a finitely generated graded Lie algebra (over $\Z$), 
with graded pieces $\g_r$, for $r\ge 1$. Then both $\g'$ and $\g''$ are 
graded sub-Lie algebras. Thus, $\g/\g''$ is in a natural 
way a graded Lie algebra, with derived subalgebra $\g'/\g''$.  
By analogy with the definition given by Chen \cite{Chen51} 
for his numerical invariants of groups (see \S\ref{subsec:lcs}), we define 
the {\em Chen ranks}\/ of $\g$ to be 
\begin{equation}
\label{eq:inf chen ranks}
\theta_r(\g)=\rank\, (\g/\g'')_r.
\end{equation}

For instance, if $\g=\Lie(\Z^n)$ is the free Lie algebra of rank $n$, the Chen groups 
$(\g/\g'')_r$ are all free abelian, of ranks $\theta_r=\theta_r(\g)$ given by 
$\theta_1=n$ and 
\begin{equation}
\label{eq:chen-free}
\theta_r=(r-1)\binom{n+r-2}{r}
\end{equation}
for $r\ge 2$, see \cite{Chen51, CS-conm95, PS-imrn04}.

Following \cite{PS-imrn04}, we associate to a graded Lie algebra $\g$ as above 
a graded module $\B(\g)$ over the symmetric algebra $S=\Sym(\g_1)$, as follows.  
The adjoint representation of $\g_1$ on $\g/\g''$ 
defines an $S$-action on $\g'/\g''$, given by $h\cdot \bar{x}=\overline{[h,x]}$, 
for $h\in \g_1$ and $x\in \g'$.  Clearly, this action is compatible with the 
grading on $\g'/\g''$.  We let the {\em infinitesimal Alexander invariant}\/ 
of $\g$ to be the graded $S$-module 
\begin{equation}
\label{eq:inf-alex}
\B(\g)=\g'/\g'' .
\end{equation}

Assume now that the graded Lie algebra $\g$ is generated 
in degree $1$, in which case $\g'=\bigoplus_{r\ge 2} \g_r$.  
Since the grading for $S$ starts in degree $0$, we are 
led to define the grading on $\B(\g)$ as 
$ \B(\g)_{r} = (\g'/\g'')_{r+2}$ for $r\ge 0$.
We then have the following `infinitesimal' version of a 
classical formula due to Massey, which relates the 
Chen ranks of $\g$ to the Hilbert series of 
$\B(\g)\otimes \Q$, viewed as a graded module 
over the polynomial ring $S\otimes \Q$.

\begin{proposition}[\cite{SW-mz17, Su-pisa24}]
\label{prop:inf Massey}
Let $\g$ be a finitely generated, graded Lie algebra generated in degree $1$.
Then the Chen ranks of $\g$ are given by
\begin{equation}
\label{eq:inf Massey}
\sum\limits_{r\ge 2}\theta_{r}(\g)\cdot t^{r-2}=\Hilb (\B(\g)\otimes \Q,t).  
\end{equation}
\end{proposition}

Applying these considerations to the case when $\g=\h(\M)$ is 
the holonomy Lie algebra of a simple matroid $\M$, we define 
the {\em holonomy Chen Lie algebra}\/ of $\M$ 
as $\h(\M)/\h(\M)''$. This is again a graded Lie algebra; we call the 
ranks of its graded pieces, 
\begin{equation}
\label{eq:holo-Chen-ranks}
\theta_r(\M) = \theta_r(\h(\M))= \rank \big(\h(\M)/\h(\M)''\big)_r ,
\end{equation}
the {\em holonomy Chen ranks}\/ of the matroid $\M$. It follows from the definitions that 
$\theta_r(\M)\le \phi_r(\M)$ for all $r\ge 1$, with equality for $r\le 3$.

\begin{example}
\label{ex:chen-U2n}
Let $\M=U_{2,n}$. 
Since $\h(\M)\cong \Lie(\Z^{n-1})\times \Lie(\Z)$,
the holonomy Chen ranks $\theta_r=\theta_r(\M)$ are given by 
$\theta_1=n$ and $\theta_r=(r-1)\binom{n+r-3}{r}$ for $r\ge 2$. 
In particular, when $n=2$, we have that $\theta_r=0$ for $r\ge 2$. 
\end{example}

\begin{example}
\label{ex:chen-graphic}
The holonomy Chen ranks $\theta_r=\theta_r(\M(\Gamma))$ of a graphic matroid 
are given by $\theta_1=\kappa_2$, $\theta_2=\kappa_3$, and 
$\theta_r = (r-1) (\kappa_3 + \kappa_4)$ for $r\ge 3$, where recall 
$\kappa_s$ counts the number of $K_s$ subgraphs of the graph $\Gamma$. 
This formula was proved in \cite{SS-tams05}, with a new proof being now given in \cite{AFRS}.
The case when $\Gamma=K_n$ was originally done in \cite{CS-conm95}, while 
the case when $\kappa_4=0$ was treated in \cite{PS-cmh06}. 
\end{example}

\subsection{The local holonomy Lie algebra}
\label{subsec:holo-loc}

In this section, we give a ``local" lower bound for the holonomy ranks 
and the holonomy Chen ranks of an arbitrary (simple) matroid $\M$. 
Let $\M_X$ be the localization of $\M$ at a rank-$2$ flat $X\in L_2(\M)$. 
We observe that the inclusion $\M_X \inj \M$ induces a (split) monomorphism 
of OS-algebras, $\OS(\M_X)\to \OS(\M)$, which in turn yields an epimorphism 
of holonomy Lie algebras, $\h(\M)\surj \h(\M_X)$. 

Let us define the {\em local 
holonomy Lie algebra}\/ of the matroid $\M$ as the product of the holonomy 
Lie algebras of the localizations at its rank-$2$ flats, 
\begin{equation}
\label{eq:holo-loc}
\h(\M)^{\loc} = \prod_{X\in L_2(\M)} \h(\M_X).
\end{equation}
The maps $\h(\M)\surj \h(\M_X)$  
assemble into a morphism $\Pi\colon \h(\M) \to \h(\M)^{\loc}$, 
which we will write simply as $\Pi\colon \h\to \h^{\loc}$ 
if the matroid $\M$ is understood. 
The next theorem extends results of Papadima and Suciu 
\cite{PS-cmh06}, from the context of hyperplane 
arrangements to arbitrary (simple) matroids.

\begin{proposition}
\label{prop:holo-loc}
Let $\M$ be a matroid with holonomy Lie algebra $\h=\h(M)$, and let $\B(\h)=\h'/\h''$ 
be its infinitesimal Alexander invariant, viewed as a module over $S=\Sym(\h_1)$. 
\begin{enumerate}[label=(\arabic*)]
\item \label{h1} 
The morphism of graded Lie algebras $\Pi\colon \h \to \h^{\loc}$  
is a surjection in degrees $r\ge 3$ and an isomorphism in degree $r=2$. 
\item  \label{h2} 
The map $\Pi$ induces an epimorphism of graded $S$-modules 
$\overline{\Pi}\colon \B(\h) \surj \B(\h^{\loc})$.
\end{enumerate}
\end{proposition}

\begin{proof}
The first statement is proved in a manner similar to the way 
\cite[Proposition~2.1]{PS-cmh06} is proved for hyperplane 
arrangements: the proof relies only on certain combinatorial properties of 
the lattice $L_{\le 2}(\M)$, which are identical for realizable and 
non-realizable simple matroids.

It follows from part \ref{h1} that the map $\Pi\colon \h\to \h^{\loc}$ 
restricts to a surjective morphism $\h' \to (\h^{\loc})'$, which in turn 
induces a surjection $\h'/\h'' \to (\h^{\loc})'/(\h^{\loc})''$ 
on abelianizations. This map coincides with the map $\overline{\Pi}$, 
and this completes the proof of part \ref{h2}.
\end{proof}

As a corollary, we obtain the following ``local" lower bounds for the holonomy 
ranks and the holonomy Chen ranks of a matroid.

\begin{corollary}
\label{cor:holo-bound}
For a matroid $\M$, write $\widetilde{L}_2(\M)\coloneqq \{X\in L_2(\M) : \abs{X}\ge 3\}$. 
The holonomy ranks $\phi_r(\M)$ and the holonomy Chen ranks $\theta_r(\M)$ 
admit the lower bounds
\begin{align}
\label{eq:phi-bound}
\phi_r(\M) &\ge \sum_{X\in \widetilde{L}_2(\M)} \phi_r(\M_X) =
\frac{1}{r} \sum_{d\mid r} \mu(d) \sum_{X\in \widetilde{L}_2(\M)}   \mu(X)^{r/d}
\\
\label{eq:theta-bound}
\theta_r(\M) &\ge \sum_{X\in \widetilde{L}_2(\M)} \theta_r(\M_X) = 
(r-1) \sum_{X\in \widetilde{L}_2(\M)} \binom{\mu(X)+r-2}{r} ,
\end{align}
valid for all $r\ge 2$, and with equality for $r=2$.
\end{corollary}

\begin{proof}
The inequality in \eqref{eq:phi-bound} is a direct consequence of 
Proposition \ref{prop:holo-loc}, part \ref{h1}, whereas the inequality 
in \eqref{eq:theta-bound} follows from Proposition \ref{prop:holo-loc}, 
part \ref{h2} and formula \eqref{eq:inf Massey}.  

To establish the equalities in \eqref{eq:phi-bound} and \eqref{eq:theta-bound}, 
first note that each localized matroid $\M_X$ is 
isomorphic to the rank-$2$ uniform matroid $U_{2,\abs{X}}$ on ground 
set of size $\abs{X}=\mu(X)+1$. The claimed equalities 
now follow from the computations in Examples \ref{ex:holo-U2n} 
and \ref{ex:chen-U2n}.
\end{proof}

The inequality \eqref{eq:phi-bound} was first proved in the realizable case 
by Falk \cite{Fa89}, and reproved in the same context in \cite{PS-cmh06,Su-decomp}. 
The inequality \eqref{eq:theta-bound} was first proved in the realizable case 
by Cohen and Suciu \cite{CS-tams99}, and reproved in the same context 
in \cite{PS-cmh06,Su-decomp}. The next example shows that both 
inequalities can be strict, even for $\phi_3(\M)=\theta_3(\M)$. 

\begin{example}
\label{ex:braid-holo}
Consider the graphic matroid $\M=\M(K_n)$.
From the formulas in Example \ref{ex:chen-graphic}, we have that 
$\theta_3(\M)=2\binom{n+1}{4}$, which is strictly greater than  
$\sum_{X\in \widetilde{L}_2(\M)} \theta_3(\M_X)=2\binom{n}{3}$ 
if $n\ge 4$.
\end{example}

\subsection{Decomposable matroids}
\label{subsec:decomp-mat}

The notion of decomposability, introduced in the realizable case by 
Papadima and Suciu in \cite{PS-cmh06} and studied in detail 
in \cite{PS-ejm20, Su-decomp}, serves as a finiteness criterion 
for the complexity of the holonomy Lie algebra of a matroid $\M$: 
roughly speaking, it detects when $\h=\h(\M)$ is determined by 
its local version, $\h^{\loc}$, at least additively.

A matroid $\M$ is said to be {\em decomposable}\/ (over $\Z$) if the map 
$\Pi\colon \h \to \h^{\loc}$ restricts to an isomorphism 
$\h_3\isom \h_3^{\loc}$ in degree $3$. Likewise, there is an induced 
morphism $\Pi\otimes \k \colon\h\otimes \k \to \h^{\loc} \otimes \k$ 
for each field $\k$, and we say $\M$ is {\em decomposable over $\k$}\/ if 
this map is an isomorphism in degree $3$. This notion only depends 
on the characteristic of the field, since $\dim_{\k}( \h_r \otimes \k)$ 
depends only on $\ch(\k)$. Moreover, $\M$ is decomposable over $\Z$ 
if and only if it is decomposable over $\k=\Q$ and $\k=\Z_p$, for all primes $p$.

For instance, a graphic matroid $\M=\M(\Gamma)$ is decomposable (over $\Z$, 
or over a field $\k$) if and only if $\Gamma$ contains no $K_4$ subgraphs, 
see \cite[Proposition 7.1]{PS-cmh06}. Other examples of decomposable 
matroids include the $\operatorname{X}_3$, $\operatorname{X}_2$, and 
non-Pappus matroids, see \cite{PS-cmh06, Su-conm01, Su-decomp}.

\begin{theorem}[\cite{PS-cmh06}]
\label{thm:ps-decomp}
Let $\M$ be a matroid which is decomposable over $\k=\Z$ or $\k$ a field. Then the 
restriction of the map $\Pi\otimes \k$ to derived Lie algebras, 
$\Pi\otimes \k\colon \h'\otimes \k\to (\h^{\loc})'\otimes \k$, 
is an isomorphism. 
\end{theorem} 
The proof given in \cite{PS-cmh06} in the realizable case works as well 
for arbitrary (simple) matroids. As an immediate application of this theorem, we 
have the following corollary.

\begin{corollary}
\label{cor:decomp}
Let $\M$ be a matroid.
\begin{enumerate}[label=(\arabic*)]
\item \label{dec1} 
If $\M$ is decomposable over $\Z$, then $\h(\M)$ is torsion-free.
\item  \label{dec2} 
If $\M$ is decomposable over $\Q$, then the inequalities from 
Corollary \ref{cor:holo-bound} hold as equalities for all $r\ge 2$.
\end{enumerate}
\end{corollary}

\section{Resonance varieties}
\label{sect:res}

The resonance varieties of a simple matroid $\M$ are algebro-geometric 
invariants that encode the structure of vanishing products in the Orlik--Solomon algebra 
$\OS(\M)$. Defined as loci in $\OS^1(\M)\otimes \k$ where the cohomology of certain cochain 
complexes jumps, the varieties $\RR^q_s(\M,\k)$ capture dependencies among circuits and flats. 
In this section, we construct them for graded commutative algebras, explore their combinatorial 
properties, including propagation, and specialize to matroids, revealing connections 
between degree-$1$ resonance varieties, multinets, and Koszul modules, 
while noting that their linearity in higher degrees 
may distinguish realizable from non-realizable matroids.

\subsection{The resonance varieties of a graded algebra}
\label{subsec:res cga}

We start by defining the resonance varieties in the more general context 
of graded algebras, as done in many works on the subject, including 
\cite{MS-aspm, PS-mrl, PS-crelle, Su-tcone, Su-edinb20, Su-conm23, 
AFRSS24, AFRS24, Su-pisa24}. 
Let $A$ be a graded, graded-commutative algebra (for short, 
a $\cga$) over a field $\k$ of characteristic different from $2$.  
Throughout, we will assume that 
\begin{itemize}[itemsep=2pt]
\item $A$ is non-negatively graded, i.e., $A=\bigoplus_{q\ge 0} A^q$. 
\item $A$ is of finite-type, i.e., each graded piece $A^q$ 
is finite-dimensional.
\item $A$ is connected, i.e., $A^0=\k$, generated by the unit $1$. 
\end{itemize} 
We will write $b_q=b_q(A)$ for the Betti numbers of $A$ 
(that is, $b_q = \dim_{\k} A^q$), and we will generally 
assume that $b_1>0$, so as to avoid trivialities. 

By graded-commutativity of the product and the assumption that $\ch \k\ne 2$, 
each element $a\in A^1$ squares to zero. (For a way to define resonance varieties in 
characteristic $2$ we refer to \cite{Su-conm23}.) 
We thus obtain a cochain complex (known in the context of OS-algebras 
as the {\em Aomoto complex}), 
\begin{equation}
\label{eq:aomoto}
\begin{tikzcd}
(A , \delta_a)\colon \quad
A^0 \arrow[r, "\delta^0_a"] & A^1 \arrow[r, "\delta^1_a"] & A^2 \arrow[r, "\delta^2_a"] & \cdots
\end{tikzcd}
\end{equation}
with differentials $\delta^q_a(u)=a\cdot u$, for all $u\in A^q$.    
The {\em resonance varieties}\/ of $A$ (in degree $q \ge 0$ 
and depth $s\ge 0$) are defined as 
\begin{equation}
\label{eq:rvs}
\RR^q_s(A)=\big\{a \in A^1 : \dim_{\k} H^q(A,\delta_a) \ge s\big\}. 
\end{equation}

In other words, the resonance varieties record the locus of points $a$ 
in the affine space $A^1=\k^{b_1}$ where the `twisted' Betti numbers 
$b_q(A,a)=\dim_{\k} H^q(A,\delta_a)$ jump by at least $s$.   
In particular, $\RR^q_0(A)=A^1$.   Here is another 
description of these sets, which follows directly from the definitions.

\begin{lemma}
\label{lem:res}
An element $a\in A^1$ belongs to $\RR^q_s(A)$ if and only if there exist 
$u_1,\dots ,u_s\in A^q$ such that $au_1=\cdots =au_s=0$ 
in $A^{q+1}$, and the set $\{au,u_1,\dots ,u_s\}$ 
is linearly independent in $A^q$, for all $u\in A^{q-1}$.
\end{lemma}

Consequently, $\RR^q_{b_q}(A)=\{0\}$ and 
$\RR^q_s(A)= \emptyset$ for $s>b_q$; in particular, 
$\RR^0_1(A) = \{ 0\}$ and $\RR^0_s(A) = \emptyset$ for $s>1$,
since $A$ is connected. Thus, for each $q \ge 0$, 
we have a descending filtration, 
\begin{equation}
\label{eq:resfilt}
A^1=\RR^q_0(A)\supseteq \RR^q_1(A)\supseteq \cdots 
\supseteq \RR^q_{b_q} (A)=\{0\}\supset \RR^q_{b_{q+1}} (A) =\emptyset.
\end{equation}
It follows that $b_q(A)=\max\{ s : 0\in \RR^q_s(A)\}$. 

Clearly, if $a$ is a non-zero element in $A^1$, then $a\in \RR^q_s(A)$ if and only if 
$\lambda a\in \RR^q_s(A)$ for all $\lambda\in\k^{*}$. Thus, the resonance 
varieties determine projective subvarieties $\P(\RR^q_s(A))$ of $\P (A^1)$.

For simplicity, we will oftentimes denote by $\RR^q(A)=\RR^q_1(A)$ 
the depth-$1$ resonance varieties of $A$. 
The most studied is the one in degree-$1$, 
which can be described as the set
\begin{equation}
\label{eq:res1}
\RR^1(A) =\big\{ a \in A^1 : \text{$\exists\, b \in A^1, \
0\neq a\wedge b \in K^{\perp} $}\big\} \cup \{0\},
\end{equation}
where $K^{\perp}$ denotes the kernel of the multiplication map $
A^1\wedge A^1 \to A^2$.

We say that a linear subspace $U\subset A^1$ is {\em isotropic}\/ 
if the restriction of the multiplication map 
$A^1\wedge A^1\to A^2$ to $U\wedge U$ is the zero map. 
If $U$ is an isotropic subspace of dimension $s>0$, then 
$U\subseteq \RR^1_{s-1}(A)$. Moreover, $\RR^1(A)$ is the union 
of all isotropic planes in $A^1$. 

\subsection{Properties of resonance}
\label{subset:natural}

The resonance varieties enjoy several pleasant properties. 
To start with, let $\varphi\colon A\to B$ be a morphism of graded $\k$-algebras. 
Then, for each $a\in A^1$, there is an induced homomorphism, 
$\varphi_a\colon H^{*}(A,\delta_a)\to H^{*}(B,\delta_{\varphi(a)})$, 
which sends a cohomology class $[b]$ represented by 
an element $b\in A^q$ such that $ab=0\in A^{q+1}$ to $[\varphi(b)]$.

\begin{proposition}[ \cite{Su-edinb20, Su-conm23}]
\label{prop:functres}
Let $\varphi\colon A\to B$ be a morphism of graded 
algebras such that $\varphi^i\colon A^q\to B^q$ is injective and 
$\varphi^{q-1}$ is surjective, for some $q\ge 1$. Then 
the homomorphisms $\varphi^i_a\colon H^{q}(A,\delta_a)\to 
H^{q}(B,\delta_{\varphi(a)})$ are injective, for all $a\in A^1$.
If, moreover, the map $\varphi^1\colon A^1\to B^1$ is 
injective.  Then this map restricts to inclusions 
$\varphi^1\colon \RR^q_s(A)\inj  \RR^q_s(B)$, for all $s\ge 0$.
\end{proposition}

\begin{corollary}
\label{cor:functorial resonance}
Let $\varphi\colon A\to B$ be a morphism of connected $\cga$s.  
If the map  $\varphi^1\colon A^1 \to B^1$ is injective, then 
$\varphi^1(\RR^1_s(A))\subseteq  \RR^1_s(B)$, for all $s\ge 0$.
\end{corollary}

In general, though, even if $\varphi\colon A\to B$ is  
an injective morphism between two graded algebras, the set 
$\varphi^1(\RR^q_s(A))$ may not be included in $\RR^q_s(B)$, 
for some $q>1$ and $s>0$.

One of the more useful properties of resonance varieties 
is the way they behave with respect to tensor products of graded algebras. 

\begin{proposition}[\cite{PS-plms17, SW-mz17, Su-edinb20, Su-conm23}]
\label{prop:resprod}
Let $A=B \otimes_{\k} C$ be the tensor product of two connected, 
finite-type graded $\k$-algebras. Then 
\begin{align*}
\RR^1_k(B \otimes_\k C)&=\RR^1_k(B)\times \{0\} \cup \{0\}\times \RR^1_k(C),\\
\RR^q(B \otimes_\k C)&=\bigcup\limits_{p\ge 0} \RR^p(B)\times  \RR^{q-p}(C), 
 \quad \textrm{ if } q\ge 2.
\end{align*}
\end{proposition}

Let us note that the resonance varieties of a graded algebra $A$ 
do not depend in an essential way on the field $\k$, but rather, just on its 
characteristic. Indeed, if $\k\subset \K$ is a field extension, then the $\k$-points 
on $\RR^q_s(A\otimes_\k \K)$ coincide with $\RR^q_s(A)$.  Thus, there is 
not much loss of generality in assuming that $\k$ is algebraically closed. 

\subsection{Equations for the resonance varieties}
\label{subsec:eqresvar}
Once again, let $A$ be a connected, finite-type $\cga$ 
over a field $\k$.  Without essential loss of generality, we will 
assume that $n\coloneqq b_1(A)$ is at least $1$. Let us pick a basis 
$\{ e_1,\dots, e_n \}$ for the $\k$-vector space 
$A^1$, and let $\{ x_1,\dots, x_n \}$ be the Kronecker 
dual basis for the dual vector space $A_1=(A^1)^{\vee}$.  
These choices allow us to identify the symmetric algebra $\Sym(A_1)$ 
with the polynomial ring $S=\k[x_1,\dots, x_n]$. 

The well-known Bernstein--Bernstein--Gelfand correspondence (see e.g.~\cite{EPY03}) 
yields a cochain complex of finitely generated, free $S$-modules, 
\begin{equation}
\label{eq:univ aomoto}
\begin{tikzcd}[column sep=20pt]
\mathbf{L}(A)=(A\otimes_{\k} S,\delta_A)\colon \ 
\cdots \arrow[r]
& A^{q-1}\otimes_{\k} S \arrow[r, "\delta^{q-1}_A"]
& A^{q} \otimes_{\k} S \arrow[r, "\delta^{q}_A"]
& A^{q+1} \otimes_{\k} S \arrow[r]
& \cdots ,
\end{tikzcd}
\end{equation}
with differentials given by $\delta^{q}_A(u \otimes 1)= \sum_{i=1}^{n} 
e_i u \otimes x_i$ for $u\in A^{q}$. By construction, the matrices associated 
to these differentials have entries that are linear forms in the variables of $S$.    

It is readily verified that the evaluation of the cochain complex $\mathbf{L}(A)$ 
at an element $a\in A^1$ coincides with the cochain complex $(A,\delta_a)$ 
from \eqref{eq:aomoto}, that is to say, $\left.\delta^q_A\right|_{x_i=a_i} = \delta^q_a$. 
By definition, an element $a \in A^1$ belongs to $\RR^q_k(A)$ 
if and only if 
\begin{equation}
\label{eq:rank delta}
\rank \delta^{q-1}_a + \rank \delta^{q}_a \le b_q(A) - s.
\end{equation}
Let $J_d(\psi)$ denote the ideal of $d\times d$ minors 
of a $p\times r$ matrix $\psi$ with entries in $S$, with 
the convention that $J_0(\psi)=S$ and $J_d(\psi)=0$ if 
$d>\min(p,r)$.  Using the standard fact that 
$J_d (\phi\oplus \psi)= \sum_{i+k=d} J_{i}(\phi ) \cdot  J_{k}(\psi)$,
we conclude that 
\begin{equation}
\label{eq:rika}
\RR^q_s(A)= V \Big( J_{b_{q}(A)-s+1} \big(\delta^{q-1}_A\oplus \delta^{q}_A\big) \Big) .
\end{equation}

The {\em resonance schemes}\/ 
$\bR^q_s(A)$ are defined by the ideals 
$J_{b_{q}(A)-s+1} \big(\delta^{q-1}_A\oplus \delta^{q}_A\big)$ 
from above and have as underlying 
sets the resonance varieties $\RR^q_s(A)$. 
The degree $1$ resonance varieties admit an even simpler description.  
Clearly, the map $\delta^0_A\colon S\to S^n$ has matrix $\big( x_1 \cdots x_n\big)$, 
and so $V\big(J_1(\delta^{0}_A)\big)=\{0\}$. It follows that 
$\RR^1_s(A)=  V ( J_{n-s} (\delta^1_A) )$ 
for $0\le s<n$ and $\RR^1_n(A)=\{0\}$. 

\subsection{Koszul modules and their support loci}
\label{subsec:kozul-res}

Let $A_i\coloneqq (A^i)^{\vee}$ and $\partial^A_i\coloneqq (\delta_A^{i-1})^{\vee}$ 
be the respective $\k$-duals, 
and consider the chain complex of finitely generated $S$-modules dual to the one 
from \eqref{eq:univ aomoto},
\begin{equation}
\label{eq:a-tensor-s}
\begin{tikzcd}[column sep=24pt]
\bigl(A_{*} \otimes_{\k} S,\partial\bigr)\colon
 \cdots \ar[r]
&A_{q+1}\otimes_{\k} S \ar[r, "\partial_{q+1}^A"]
&A_{q} \otimes_{\k} S \ar[r, "\partial_{q}^A"]
&A_{q-1} \otimes_{\k} S \ar[r]
& \cdots.
\end{tikzcd}
\end{equation}
Following \cite{PS-mrl, Su-tcone, AFRSS24}, 
we define the {\em Koszul modules}\/  (in degrees $q \ge 0$) of the graded 
algebra $A$ as the homology $S$-modules of this chain complex, that is,
\begin{equation}
\label{eq:wi-a}
W_q(A) \coloneqq H_q\bigl(A_{*}\otimes_{\k} S\bigr) .
\end{equation}

It is clear that the Koszul modules are finitely generated, graded $S$-modules 
and that $W_0(A)=\k$ is the trivial $S$-module.
Setting $E_{*}\coloneqq \bwedge A_1$, the first Koszul module admits
the presentation
\begin{equation}
\label{eq:pres-w1a}
\begin{tikzcd}[column sep=18pt]
\big(E_3 \oplus K \big)  \otimes_{\k}  S
\ar[r, "\partial_3^E +\iota  \otimes_{\k} \id_S"] &[32pt]
E_2  \otimes_{\k}  S \ar[r, two heads] & W_1(A),
\end{tikzcd}
\end{equation}
where
$\begin{tikzcd}[column sep=18pt]
\hspace*{-5pt} K=\bigl\{\varphi \in A_1\wedge A_1=(A^1\wedge A^1)^{\vee} : 
\varphi_{|K^{\perp}}\equiv 0\bigr\}  \ar[r, hook, "\iota"] & A_1\wedge A_1 =E_2
\end{tikzcd}$. 

The {\em resonance schemes}\/ (in degree $q$ and depth $s$) of the $\cga$ 
$A$ are defined by the annihilator ideals of the exterior powers of the 
Koszul modules of $A$,
\begin{equation}
\label{eq:ria-spec}
\bR_{q,s}(A) \coloneqq \Spec \bigl(S/ \ann \big(\bwedge^s W_q(A)\big)\bigr).
\end{equation}
We will denote by $\RR_{q,s}(A)=\supp\big( \bwedge^s W_q(A)\big)$ 
the underlying sets, and call them the {\em support resonance loci} of $A$. For simplicity,  
we will write $\RR_{q}(A)=\RR_{q,1}(A)$ for the depth-$1$ support loci, and $\bR_{q}(A)$ 
for the corresponding schemes. Clearly $\RR_0(A)=\RR^0(A)=\{0\}$. 

More generally, suppose $W_j(A)\neq 0$ for all $1\le j\le q$; 
then, as shown in \cite{PS-mrl}, the degree-$1$ support resonance loci
are related to the degree-$1$ jump resonance loci by the equality
\begin{equation}
\label{eq:union-res}
\bigcup_{j\le q} \RR_j(A) = \bigcup_{j\le q} \RR^j(A).
\end{equation}
In particular, if $W_1(A)\ne 0$, then $\RR_1(A) =\RR^1(A)$.

\subsection{Resonance varieties of matroids}
\label{subsec:res-OS}

Let $\M$ be a rank-$\ell$ simple matroid on ground set $[n]$ 
and let $A=A(\M)$ be its Orlik--Solomon algebra. Recall that $A=E/I$, 
where $E=\bwedge V$ is the exterior algebra 
on the free abelian group $V$ with basis $\{e_1,\dots , e_n\}$ and $I$ 
is the ideal generated by $\big\{ \partial(e_S) : \text{$S \subseteq [n]$ 
is a circuit of $\M$} \big\}$. 

Fix a ground field $\k$ and consider the Orlik--Solomon algebra over this field, 
$\OS_{\k}(\M)=\OS(\M)\otimes \k$. Clearly, this is a connected, finite-type $\k$-$\cga$. 
The resonance varieties of the matroid $\M$ (over $\k$) are defined as the resonance 
varieties of its OS-algebra over $\k$, 
\begin{equation}
\label{eq:res-os}
\RR^q_s(\M,\k) = \RR^q_s(\OS_{\k}(\M) ) . 
\end{equation}
Since the integral OS-algebra $A$ is torsion-free, we have that $a^2 = 0$ in 
$A^2$ for every $a\in A^1$, and so the definition makes sense even when $\ch(\k)=2$. 
By construction, the sets $\RR^q_s(\M,\k)$ are homogeneous subvarieties 
of the affine space $V_{\k}\coloneqq V\otimes \k =\OS^1_{\k}(\M)$ and 
$\RR^q_s(\M,\k)\subseteq \RR^q_1(\M,\k)$, for all $s\ge 1$.  
We will often abbreviate $\RR^q(\M,\k)\coloneqq \RR^q_1(\M,\k)$. 

Suppose a matroid decomposes as $\M=\M_1\oplus \M_2$. Then, as noted 
in \S\ref{subsec:OS}, its OS-algebra decomposes as 
$\OS(\M)\cong \OS(\M_1)\otimes \OS(\M_2)$, 
and thus, by Proposition \ref{prop:resprod}, 
$\RR^k(\M,\k)=\bigcup_{p+q=k}\RR^p(\M_1,\k)\times \RR^q(\M_2,\k)$. 

The next theorem puts together several structural results about the resonance varieties  
of matroids, extracted from the works of Yuzvinsky \cite{Yu95}, Eisenbud--Popescu--Yuzvinsky 
\cite{EPY03}, Budur \cite{Bu11}, and Denham \cite{De16}. In the original references, 
the results in items \ref{r1}--\ref{r3} are stated only in the case when $\M$ is realizable, 
but the proofs are homological in nature, and do not rely on realizability or geometric 
input (this is also noted in \cite{De16} for items \ref{r1} and \ref{r2}).  
Moreover, the result from item \ref{r2} is stated only for $\k=\C$, but once again 
the proof works for arbitrary fields. 

\begin{theorem}
\label{thm:res-matroids}
Let $\M$ be a rank-$\ell$ simple matroid on $n$ points, and let 
$\k$ be a field. Then,
\begin{enumerate}[label=(\arabic*)] 
\item \label{r1}
\cite{Yu95} 
The resonance varieties $\RR^q_s(\M,\k)$  lie inside the hyperplane 
$\overline{V}_{\k}=\big\{ x\in V_{\k} : \sum_{i=1}^n x_i=0\big\}$. 

\item  \label{r2}
\cite{EPY03}  
The depth-$1$ resonance varieties {\em propagate}, that is, 
$\RR^q_1(\M,\k) \subseteq \RR^{q+1}_1(\M,\k)$ for $0\le q \le \ell-1$. 

\item  \label{r3}
\cite{Bu11} 
More specifically,  
$\RR^q_1(\M,\k)\subseteqq \RR^{q+1}_2 (\M,\k)$ 
for $q\le \ell-2$ and $\RR^q_1(\M,\k)\subseteqq \RR^{q+1}_{s} (\M,\k)$ 
for $q< \ell-2$ and $s\le 1+ \frac{\ell-3}{q+1}$.

\item  \label{r4}
\cite{Yu95} 
If $\M$ is connected, then 
$\RR^{\ell-1}_1(A_{\k}(\M))  = \RR^{\ell}_1(A_{\k}(\M)) = \overline{V}_{\k}$. 

\item  \label{r5}
\cite{De16} 
If $\ch(\k)\nmid n$, then $\RR^q_1(\overline{A}_{\k}(\M)) = 
\RR^q_1(A_{\k}(\M))$ for $0\le q\le \ell-1$.

\end{enumerate}
\end{theorem}

As an immediate consequence of this theorem, we obtain the following corollary.

\begin{corollary}
\label{cor:res-mat}
Let $\M$ be a simple, connected, rank-$\ell$ matroid on $n$ points. Then,  
for any field $\k$, 
\begin{equation}
\label{eq:propagate}
\{0\}= \RR^0_1(\M,\k)\subseteq \RR^1_1(\M,\k)\subseteq\cdots\subseteq 
\RR^{\ell-1}_1(\M,\k) =\RR^{\ell}_1(\M,\k)= \overline{V}_{\k}. 
\end{equation}
\end{corollary}

\begin{example}
\label{ex:res-unif}
Let $\M=U_{n,n}$ be the free matroid ground set $[n]$; then 
$\OS(\M)=E$ and so $\RR^q_s(\M,\k)$ is equal to $\{0\}$ if $0\le s\le \binom{n}{q}$ 
and is empty otherwise. On the other hand, if $\M=U_{\ell,n}$ is the uniform rank-$\ell$ 
matroid on $n$ points, and if $\ell < n$, then $\RR^q_1(\M,\k)=\{0\}$ for $0\le q\le \ell-1$ and 
$\RR^{\ell-1}_1(\M,\k)=\overline{V}_{\k}$, since $\M$ is connected.
\end{example}

As we shall see in \S\ref{sect:hyp-arr}, all irreducible components of the 
resonance varieties $\RR^q_s(\M,\k)$ are linear subspaces of $ \overline{V}_{\k}$, 
provided the matroid $\M$ is realizable by a hyperplane arrangement and 
$\ch(\k)=0$. Furthermore, as we shall see in \S\ref{subsec:r1}, 
linearity holds in the non-realizable case when $q=1$ and again $\ch(\k)=0$. 

Using results of Schechtman and Varchenko \cite{SV91}, Denham constructed in 
\cite[Theorem~5]{De16}  a ``linear envelope" for the resonance varieties of arbitrary matroids, 
in all degrees. To describe his result, we need some notation. Given  
a partition $\pi=(\pi_1,\dots , \pi_k)$ of $[n]$ with $k$ parts, let 
$P_{\pi}$ denote the codimension-$k$ subspace of $\k^n$ given by 
\begin{equation}
\label{eq:P-pi}
P_{\pi}=\Big\{ x\in \k^n : \sum_{j\in \pi_i} x_j = 0\ \text{ for $1\le i\le k$}\Big\}.
\end{equation}
For instance, if $k=1$, then $P_{\pi}=\overline{V}_\k$, while if $k=n$, then $P_{\pi}=\{0\}$.

\begin{theorem}[\cite{De16}]
\label{thm:res-bound}
Let $\M$ be a matroid of rank $\ell$ on $n$ points, and let $\k$ be an arbitrary field. Then, 
for all $0\le q\le \ell$,
\[
\RR^q(\M,\k) \subseteq \bigcup_{i=1}^{q+1} 
\bigcup_{X\in L_i^{\irr}(\M)} P_{\{X,[n]\setminus X\}}.
\]
\end{theorem}

\subsection{Resonance in degree $1$}
\label{subsec:r1}
If $\M'\subset \M$ is a proper sub-matroid, there is a natural epimorphism 
$\OS_{\k}(\M) \surj \OS_{\k}(\M')$, which restricts in degree $1$ to an embedding 
$\RR^1_s(\M',\k) \inj \RR^1_s(\M,\k)$ for all $s\ge 1$. The irreducible 
components of $\RR^1_s(\M,\k)$ that  lie in the image 
of such an embedding are called {\em non-essential}; 
the remaining components are called {\em essential}. 
For instance, if $X$ is a flat in $L_2(\M)$, then 
$\RR^1_s(\M_X,\k)=\overline{V}_{\k}$ for all $s\le \mu(X)$; 
the irreducible components of $\RR^1_s(\M,\k)$ 
that lie in the image of $\RR^1_s(\M_X,\k) \inj \RR^1_s(\M,\k)$ 
are called {\em local}\/ components.

The resonance varieties of a realizable matroid were first 
defined and studied by Falk \cite{Fa97} for $\k=\C$, and by 
Matei--Suciu \cite{MS-aspm} and Falk \cite{Fa07} 
for arbitrary fields. A complete description of the 
resonance varieties $\RR^1_s(\M,\C)$ was given by 
Falk and Yuzvinsky \cite{FY07} in the realizable case 
and by Marco Buzun\'ariz \cite{MB09} in general. 

Let $\M$ be a (simple) matroid on ground set $\cE$. 
Given a $k$-multinet $\NN$ on $\M$ with parts 
$(\cE_1,\dots ,\cE_k)$ and multiplicity vector $m$, 
set $w_i=\sum_{u\in \cE_i} m_u e_u$ for each $1\le i\le k$, 
and  put 
\begin{equation}
\label{eq:pm}
P_\NN=\spn \{w_2-w_1,\dots , w_k-w_1\}.  
\end{equation}
By construction, $P_\NN$ is a linear subspace of $\overline{V}_{\C}$.  
As shown in \cite{FY07, MB09}, the subspace $P_\NN$ 
lies inside $\RR^1(\M,\C)$ and has dimension $k-1$. 

Now suppose there is a sub-matroid $\M'\subset \M$ 
which supports a $k$-multinet $\NN$. By the above, 
the linear space $P_{\NN}$ lies inside $\RR^1(\M',\C)$; moreover, 
the inclusion $\M'\inj \M$ induces an embedding 
$\RR^1(\M',\C) \inj \RR^1(\M,\C)$.  Thus, $P_{\NN}$ is a 
linear space of dimension $k-1$ that lies inside $\RR^1(\M,\C)$.
Conversely, it is shown in \cite[Theorem 2.5]{FY07} (in the realizable case) 
and \cite[Theorem 3]{MB09} (in general) that all non-zero irreducible 
components of $\RR^1(\M,\C)$ arise in this fashion.  

\begin{theorem}[\cite{FY07, MB09}]
\label{thm:res fy}
The positive-dimensional, irreducible components of 
$\RR^1(\M, \C)$ are linear subspaces of $\overline{V}_{\C}$ 
that are in one-to-one correspondence 
with the multinets on sub-matroids $\M'\subset \M$, with each 
$k$-multinet $\NN$ on $\M'$ corresponding to the $(k-1)$-dimensional 
subspace $P_{\NN}$; that is,
\[
\RR^1(\M, \C) = \{0\}\cup \bigcup_{\M' \subset \M}
\bigcup_{\textnormal{$\NN$ a multinet on $\M'$}} P_{\NN}.
\]
Moreover, any two of these components intersect only at $0$.
\end{theorem}
This result can be sharpened to describe the decomposition into 
irreducible components of the $\RR^1_s(\M,\C)$, for all $s\ge 1$. 
As shown in \cite[Lemma 4]{MB09}, we have that 
\begin{equation}
\label{eq:rsa}
\RR^1_s(\M,\C) = \{0\}\cup \bigcup_{\M' \subset \M}
\bigcup_{\substack{\text{$\NN$ a multinet on $\M'$}\\
\text{with at least $s+2$ parts}}} P_{\NN}.
\end{equation}

If the matroid is decomposable over $\Q$, all resonance is local. More precisely, 
we have the following theorem, which was proved in \cite{Su-decomp} in the 
realizable case, and whose proof works in general.

\begin{theorem}[\cite{Su-decomp}]
\label{thm:decomp-res}
Let $\M$ be a simple matroid on ground set $[n]$. If $\M$ is decomposable over $\Q$, 
then, for each $s\ge 1$, 
\[
\RR^1_s(\M,\C)=\bigcup_{\substack{X \in L_2(\M)\\ \mu(X)>s}}  L_X,
\]
where $L_X = \Big\{ x \in \C^n : \sum_{i\in X} x_i =0 \ \text{and $x_i= 0$ if $i\notin X$} \Big\}$.
\end{theorem}

It is important to note that over fields $\k$ of positive characteristic, 
the resonance varieties $\RR^1(\M,\k)$ are no longer guaranteed 
to be unions of linear subspaces, even when $\M$ is realizable.
An example illustrating this non-linearity was given by Falk in \cite{Fa07}.

\subsection{Koszul modules and holonomy Chen ranks}
\label{subsec:koszul-chen}
The Koszul modules of a matroid $\M$ (over a field $\k$) are the 
Koszul modules of its Orlik--Solomon $\k$-algebra: $W_q(\M,\k)\coloneqq W_q(A_{\k}(\M))$. 
These objects are finitely generated graded modules over the polynomial ring 
$S_{\k}=\Sym (V^{\vee}\otimes \k)$. 
The first in this sequence of modules admits a fruitful  
interpretation in terms of the holonomy 
Lie algebra of $\M$, which we record in the next lemma.

\begin{lemma}
\label{lem:alex-koszul}
The Koszul module $W_1(\M,\k)$ is isomorphic (as a graded 
$S_{\k}$-module) to the infinitesimal 
Alexander invariant $\B(\h(\M)) \otimes \k$. 
\end{lemma}

\begin{proof}
The $S_{\k}$-module $W_1(\M,\k)$ has the presentation 
given in \eqref{eq:pres-w1a}. Comparing this presentation 
to that of the $S$-module $\B(\h(\M))$ 
given in \cite[Theorem 6.2]{PS-imrn04} yields the desired isomorphism.
\end{proof} 

\begin{problem}
\label{prob:high-koszul}
Given a matroid $\M$, find explicit presentations for the Koszul modules 
$W_q(\M,\k)$ for $q\ge 2$.
\end{problem}

The support resonance varieties of $\M$ (over $\k$) are 
defined as $\RR_{q,s}(\M,\k) \coloneqq \RR_{q,s}(\OS_{\k}(\M) )$. 
In particular, the sets $\RR_{q}(\M,\k)=\RR_{q,1}(\M,\k)$ 
are the support loci of the Koszul modules $W_q(\M,\k)$. 
From formula \eqref{eq:union-res}, we know that $\RR_1(\M,\k) =\RR^1(\M,\k)$ 
if $W_1(\M,\k)\ne 0$, a condition which holds precisely when $L_2(\M)$ contains 
flats of length at least $3$. 

\begin{corollary}
\label{cor:supp-bhm}
If $\widetilde{L}_2(\M)\ne \emptyset$, then the support of $\B(\h(\M)) \otimes \k$ 
coincides with the resonance variety $\RR^1(\M,\k)$.
\end{corollary}

We also know from Theorem \ref{thm:res fy}
that each positive-dimensional, irreducible component of 
$\RR^1(\M, \C)$ is a $(k-1)$-dimensional linear subspace 
of the form $P_{\NN}$ corresponding to a $k$-multinet $\NN$ 
on a sub-matroid $\M'\subset \M$.  On the other hand, we 
know from Proposition \ref{prop:inf Massey} that 
\begin{equation}
\label{eq:Massey-holo}
\sum\limits_{r\ge 2}\theta_{r}(\M)\cdot t^{r-2}=\Hilb (\B(\h(\M))\otimes \C,t).  
\end{equation}
Finally, we know from Example \ref{ex:chen-U2n} that the 
Chen ranks of the rank-$2$ uniform matroid on $k$ points are given by 
$\theta_r(U_{2,k})=(r-1)\binom{k+r-3}{r}$ for $r\ge 2$. 

These considerations lead to the following conjecture, 
that generalizes the Chen Ranks 
conjecture for arrangement groups from \cite{Su-conm01}.

\begin{conjecture}
\label{conj:chen-ranks-mat}
The holonomy Chen ranks of a simple matroid $\M$ 
are given by 
\[
\theta_r(\M) = (r-1) \sum_{\M'\subset \M} \sum_{k\ge 3} 
n_k(\M')\,  \binom{k+r-3}{r} 
\]
for $r\gg 0$, where $n_k(\M')$ is the number of $k$-multinets 
supported by a sub-matroid $\M'\subset \M$.
\end{conjecture}

In the realizable case, the conjecture has been verified by Cohen and Schenck \cite{CSc-15} 
in the case when all (non-local) multinets on the arrangement are essential. 
A key ingredient in the proof is establishing 
that the scheme $\bR_1(\M,\C)$ is projectively reduced, see \cite{CSc-15} and the 
generalization from \cite{AFRS24}. An effective version of the Chen Ranks conjecture 
is being considered in \cite{AFRS}.

\section{Hyperplane arrangements}
\label{sect:hyp-arr}
Hyperplane arrangements are geometric realizations of matroids, whose 
combinatorial structure is encoded by their intersection lattice. In this section, 
we explore the topology of their complements, focusing on the fundamental 
group and its associated Lie algebras, and examine the resonance varieties 
of the Orlik--Solomon algebra, revealing connections between degree-$1$ 
resonance, multinets, and the matroid’s flats. These invariants bridge the 
combinatorics of arrangements with their topological properties, while 
highlighting possible differences, such as the behavior of Lie algebra maps in 
positive characteristic, or the structure of the higher-degrees resonance varieties.

\subsection{Intersection lattice and complement}
\label{subsec:arr-lat}

An {\em arrangement of hyperplanes}\/ is a finite set $\A$ of hyperplanes 
(i.e., codimension-$1$ linear subspaces) in a finite-dimensional complex 
vector space $\C^{\ell}$. Each hyperplane $H\in \A$ may be viewed as the
kernel of a nontrivial linear form $\alpha_H\colon \C^{\ell}\to \C$, unique 
up to scalar. The combinatorics 
of the arrangement is encoded in its {\em intersection lattice}, $L(\A)$— 
that is, the poset of all intersections of hyperplanes in $\A$ (also 
known as {\em flats}), ordered by reverse inclusion and ranked by 
codimension. The join of two flats $X,Y\in L(\A)$ is 
$X\vee Y=X\cap Y$, while the meet $X\wedge Y$ 
is the smallest flat containing the sum $X+Y$.

We may view $\A$ as a realization of a simple matroid $\M$, whose 
points correspond to the hyperplanes in $\A$, and whose dependent 
subsets are those whose associated normal vectors are linearly dependent. 
The lattice of flats of the matroid $\M$ then identifies naturally with $L(\A)$, 
and under this dictionary, the operations of join, meet, and rank coincide. 
(For instance, the rank of a flat $X$ is both its codimension in $\C^\ell$ and 
the rank of the corresponding flat in the matroid.)

The union of the hyperplanes comprising the arrangement, $W=\bigcup_{H\in\A} H$, 
is the variety in $\C^{\ell}$ defined by the polynomial $Q=\prod_{H\in \A} \alpha_H$. 
The complement of the arrangement, $M(\A)=\C^{\ell}\setminus W$, 
is a connected, smooth, complex quasi-projective variety. Moreover,  
$M=M(\A)$ is a Stein manifold, and thus has the homotopy type of a 
CW-complex of dimension at most $\ell$. In fact, $M$ splits off the 
$\C$-linear subspace $\bigcap_{H\in \A} H$, whose dimension we call the 
{\em corank}\/ of $\A$. Accordingly, we define the {\em rank} of $\A$ as 
\begin{equation}
\label{eq:rankA}
\rank(\A)\coloneqq \ell - \corank(\A) = \dim_{\C}\big(\!\spn \{\alpha_H : H\in\A\}\big),
\end{equation}
which agrees with the rank of the associated matroid $\M$. If $\corank(\A)=0$, 
that is, if the hyperplanes intersect only at the origin, we say that $\A$ is {\em essential}.

The group $\C^*$ acts freely on $\C^{\ell}\setminus \{0\}$ via 
$\zeta\cdot (z_1,\dots,z_{\ell})=(\zeta z_1,\dots, \zeta z_{\ell})$. 
The orbit space is the complex projective space of dimension $\ell-1$, 
while the orbit map, $\pi\colon \C^{\ell}\setminus \{0\} \to \CP^{\ell-1}$,  
$z \mapsto [z]$, is the Hopf fibration. 
The set $\P(\A)=\{\pi(H)\colon H\in \A\}$ is an 
arrangement of codimension $1$ projective subspaces in $\CP^{\ell-1}$. 
Its complement, $U=U(\A)$, coincides with the quotient $\P(M)=M/\C^*$. 
The Hopf map restricts to a bundle map, $\pi\colon M\to U$, with fiber 
$\C^{*}$. Fixing a hyperplane $H_0\in \A$, we see that $\pi$ is 
also the restriction to $M$ of the bundle map 
$\C^{\ell}\setminus H_0\to \CP^{\ell-1} \setminus \pi(H_0) \cong \C^{\ell-1}$.  
This latter bundle is trivial, and so we have a diffeomorphism  
$M \cong U\times \C^*$.

The cohomology ring of a hyperplane arrangement 
complement $M=M(\A)$ was computed by Brieskorn in \cite{Br}, 
building on the work of Arnol'd on the cohomology 
ring of the pure braid group. In \cite{OS}, Orlik and Solomon 
described this ring purely combinatorially, showing it is 
determined by the intersection lattice $L(\A)$. More precisely, 
in the terminology from \S\ref{subsec:OS}, they showed 
that $H^*(M;\Z)$ is isomorphic to the OS-algebra $\OS(\M)$ of 
the matroid $\M=\M(\A)$ corresponding to $\A$. It follows from their 
description that the cohomology ring of the projectivized 
complement, $H^*(U,\Z)$, is isomorphic to the projective 
OS-algebra $\overline{\OS}(\M)$.

For every arrangement $\A$, the complement $M=M(\A)$ is a (rationally) 
formal space, in the sense of Sullivan \cite{Sullivan}.  
Indeed, for each $H\in \A$, the $1$-form 
$\omega_H= \frac{1}{2\pi \ii} \, d \log \alpha_H$ on $\C^{\ell}$ restricts 
to a $1$-form on $M$.  As shown by Brieskorn \cite{Br},  
if $\mathcal{D}$ denotes the subalgebra 
of the de~Rham algebra $\Omega^*_{\rm dR}(M)$ 
generated over $\R$ by these $1$-forms, 
the correspondence $\omega_H \mapsto [\omega_H]$ 
induces an isomorphism $\mathcal{D} \to H^*(M;\R)$, 
and then Sullivan's machinery implies that $M$ is formal over $\R$ 
(and thus, over $\Q$). 
On the other hand, as shown by Matei  \cite{Mt06}, arrangements 
complements need not be formal over a field $\k$ of characteristic $p>2$, 
due to the presence of non-vanishing Massey triple products in $H^2(M; \k)$. 

\subsection{Fundamental group}
\label{subsec:pi1}

Fix a basepoint $x_0$ in $M(\A)$, and consider the 
fundamental group $G(\A)=\pi_1(M(\A),x_0)$. 
For each hyperplane $H\in \A$, pick a meridian curve about $H$, 
oriented compatibly with the complex orientations on $\C^{\ell}$ and $H$, 
and let $\gamma_H$ denote the based homotopy class of this curve, 
joined to the basepoint by a path in $M$.  By the van Kampen theorem, 
then, the group $G=G(\A)$ is generated by the set 
$\{\gamma_H\}_{H\in \A}$. Using the braid monodromy algorithm from 
\cite{CS-cmh97}, one obtains a finite presentation 
of the form $G=F_n/R$, where $F_n$ is the free group of rank $n=\abs{\A}$  
and the relators in $R$ belong to the commutator 
subgroup $F_n'$. Consequently, the abelianization of the arrangement 
group, $G_{\ab}=H_1(G;\Z)$, is isomorphic to $\Z^n$. 
Under the diffeomorphism $M\cong U\times \C^{*}$, the arrangement group 
splits as $\pi_1(M)\cong \pi_1(U)\times \Z$, with the $\Z$ factor 
corresponding to the class of a curve looping once around all hyperplanes. 

\begin{example}
\label{ex:config}
Let $\A=\A(K_n)$ be the braid arrangement from Example \ref{ex:graphic-mat}. 
The intersection lattice is the lattice of partitions of the set $[n]$, 
with the opposite order.  The complement $M$ is the configuration space 
of $n$ ordered points in $\C$, which is a classifying space for 
the pure braid group on $n$ strings, $P_{n}$. 
\end{example} 

For the purpose of computing the group $G(\A)=\pi_1(M(\A))$, it is enough to assume 
that the arrangement $\A$ lives in $\C^3$, in which case $\bar{\A}=\P(\A)$ 
is an arrangement of (projective) lines in $\CP^2$. This is clear when the 
rank of $\A$ is at most $3$, and may be achieved otherwise 
by taking a generic $3$-slice, an operation which does not
change either the poset $L_{\le 2}(\A)$, or the group $G(\A)$, 
by the Lefschetz-type theorem of Hamm and L\^{e}.
For a rank-$3$ arrangement, the set $L_1(\A)$ is in one-to-one correspondence 
with the lines of $\bar{\A}$, while $L_2(\A)$ is in $1$-to-$1$ correspondence 
with the intersection points of $\bar{\A}$.  Moreover, the poset structure of $L_{\le 2}(\A)$ 
mirrors the incidence structure of the point-line configuration $\bar{\A}$.   

\subsection{Lie algebras associated to groups}
\label{subsec:lcs}
Before proceeding, we review several constructions that associate 
to a group $G$ some inter-connected graded Lie algebras. For 
simplicity, and since is the context relevant to us here, we 
will assume that $G$ is finitely generated, and that its abelianization 
$G_{\ab}$ is torsion-free.

Given subgroups $H$ and $K$ of $G$, 
define their commutator, $[H,K]$, to be the subgroup of $G$ 
generated by all elements of the form $[a,b]=aba^{-1}b^{-1}$ with 
$a\in H$ and $b\in K$. The lower central series (LCS) of $G$ is 
defined inductively by setting $\gamma_1(G)=G$ and 
$\gamma_{r+1}(G)=[G,\gamma_r(G)]$. It is readily verified 
that $[\gamma_r(G),\gamma_s(G)]\subset \gamma_{r+s}(G)$ for all $r,s\ge 1$. It 
follows that the terms of the series are normal subgroups of $G$. 
Furthermore, each quotient group $\gr_r(G)\coloneqq \gamma_r(G)/\gamma_{r+1}(G)$ 
lies in the center of $G/\gamma_{r+1}(G)$, and thus is abelian. 
The direct sum of these quotients, 
\begin{equation}
\label{eq:gr-G}
\gr(G) \coloneqq \bigoplus_{r\ge 1} \gr_r(G),
\end{equation}
acquires the structure of a graded Lie algebra, with the Lie bracket map 
$[\:,\:]\colon \gr_{r}(G) \otimes \gr_{s}(G)\to \gr_{r+s}(G)$ is induced 
from the group commutator.
By construction, the {\em associated graded Lie algebra}\/ $\gr(G)$ 
is generated by its degree $1$ piece, $\gr_1(G)=G_{\ab}$. 
Therefore, since $G_{\ab}$ is finitely generated, the LCS quotients 
$\gr_r(G)$ are finitely generated abelian groups; we will denote  
by $\phi_r(G)$ their ranks.

Replacing in this construction the group $G$ by its maximal 
metabelian quotient, $G/G''$, leads to the {\em Chen Lie algebra}\/  
$\gr(G/G'')$ and the Chen ranks $\theta_r(G)\coloneqq \rank \gr_r(G/G'')$. 
It follows from the definitions that $\theta_r(G)\le \phi_r(G)$ for all $r\ge 1$, with 
equality for $r\le 3$. 

Proceeding in a manner similar to that in 
\S\ref{subsec:holo}, we also consider the {\em holonomy Lie algebra}\/ 
of $G$. This is a finitely presented, quadratic Lie algebra, defined as 
\begin{equation}
\label{eq:holo-G}
\h(G) \coloneqq \Lie (G_{\ab})/ \ideal \big(\! \im (\cup_G^{\vee}) \big),
\end{equation}
where $\cup_G^{\vee}\colon H_2(G;\Z)\to H_1(G;\Z)\wedge H_1(G;\Z)=G_{\ab}\wedge G_{\ab}$ 
is the $\Z$-dual of the cup-product map $H^1(G;\Z)\wedge H^1(G;\Z) \to H^2(G;\Z)$. 
Alternatively, $\cup_G^{\vee}$ may be identified with the homomorphism $\ab_*\colon 
H_2(G;\Z)\to H_2(G_{\ab};\Z)$ induced in homology by the abelianization map, 
$\ab\colon G\to G_{\ab}$. 

The rational holonomy Lie algebra of a finite CW-complex $X$ 
(with fundamental group $G$) was first defined by Chen in \cite{Chen73}. 
A notable fact about the holonomy Lie algebra 
is its close relationship to the associated graded Lie algebra of $G$. 
In \cite{MP92}, Markl and Papadima showed that the canonical 
projection $\Lie(G_{\ab})\surj \gr(G)$ descends to an epimorphism 
of graded Lie algebras $\Psi\colon \h(G) \surj \gr(G)$, which is an  
isomorphism in degrees $1$ and $2$ (see \cite{SW-jpaa19, SW-forum} for 
more on this). In \cite{PS-imrn04}, Papadima 
and Suciu showed that the map $\Psi$ further descends to an epimorphism 
$\overline{\Psi}\colon \h(G)/\h(G)''\surj \gr(G/G'')$.  

When the group $G$ is $1$-formal, the maps 
$\Psi\otimes \Q$  and $\overline{\Psi} \otimes \Q$ 
are all isomorphisms, as shown in \cite{Sullivan} and \cite{PS-imrn04}, 
respectively. 
In general, though, the map $\Psi\otimes \Q$ fails to be injective, even in 
degree $3$. Nevertheless, as shown by Porter and Suciu \cite{PS-ejm20},  
if the map $\cup_G^{\vee}$ is injective, then the map 
$\Psi_3\colon \h_3(G) \to \gr_3(G)$ is an isomorphism. 

\subsection{Lie algebras associated to arrangements}
\label{susbec:lie-arr}

Let $\A$ be a hyperplane arrangement and let $G=G(\A)$ be the 
fundamental group of the complement $M=M(\A)$.  All the results 
relating the holonomy Lie algebra $\h(G)$ to the associated graded Lie 
algebra $\gr(G)$ mentioned in \S\ref{subsec:lcs} apply in this setting. 
Indeed:
\begin{itemize}
\item  The group $G$ is finitely presented. Moreover, its abelianization 
$G_{\ab}$ is isomorphic to $\Z^n$, where $n=\abs{\A}$, and thus is 
torsion-free. 

\item Since the space $M$ is formal, the group $G=\pi_1(M)$ is $1$-formal. 

\item 
Any classifying map $M\to K(G,1)$ induces a ring isomorphism 
$H^{\le 2}(G;\Z) \isom H^{\le 2} (M;\Z)$, see \cite{Ra97, MS-aspm}. 
Since the ring 
$H^*(M;\Z)$ is generated in degree $1$, it follows that the map 
$\cup_G\colon  H^1(G;\Z)\wedge H^1(G;\Z)\to H^2 (G;\Z)$ is 
surjective, and thus its $\Z$-dual, $\cup_G^{\vee}$, is injective. 
\end{itemize}

Now let $\M=\M(\A)$ be the matroid associated to $\A$, and let $\OS=\OS(\M)$ 
be its Orlik--Solomon algebra. From the description of the cohomology algebra 
of the complement $M=M(\A)$ given in \cite{OS}, it follows that 
$H^{\le 2}(M;\Z) \cong \OS^{\le 2}$, 
and thus $H^{\le 2}(G;\Z)\cong \OS^{\le 2}$. Therefore, $\h(G)$ is isomorphic 
to $\h(\M)$, and admits Kohno's presentation from \eqref{eq:holo-pres}. 

Putting these observations together, we obtain the following result.
\begin{theorem}
\label{thm:comparison_map}
Let $\A$ be a hyperplane arrangement, let $\M=\M(\A)$ be its associated matroid, 
and let $G=G(\A)$ be the arrangement group. Then:
\begin{enumerate}
\item \label{epi1}
There is an epimorphism of graded Lie algebras, $\Psi \colon \h(\M) \surj \gr(G)$ 
which induces isomorphisms in degrees $1$, $2$, and $3$. 
\item  \label{epi2}
The map $\Psi$ descends to an epimorphism 
$\overline{\Psi}\colon \h(\M)/\h(\M)'' \surj \gr(G/G'')$
which induces isomorphisms in degrees $1$, $2$, and $3$. 
\item  \label{epi3}
The maps $\Psi\otimes\Q$ and $\overline{\Psi}\otimes\Q$ 
are both isomorphisms.
\end{enumerate}
\end{theorem}

As a consequence, the Lie algebras $\gr(G)\otimes \Q$ and $\gr(G/G'')\otimes \Q$ 
are determined by the (truncated) intersection lattice $L_{\le 2}(\A)=L_{\le 2}(\M)$. 
Moreover, the LCS ranks $\phi_r(G)=\dim_{\Q} \gr_r(G)\otimes \Q$ and 
the Chen ranks $\theta_r(G)=\dim_{\Q} \gr_r(G/G'')\otimes \Q$ are equal 
to the corresponding holonomy ranks $\phi_r(\M) = \dim_{\Q} \h_r(\M)\otimes \Q$ and 
holonomy Chen ranks $\theta_r(\M) = \dim_{\Q} \big(\h(\M)/\h(\M)''\big)_r \otimes \Q$.

The above theorem naturally raises the following question.

\begin{problem}
\label{prob:holo-lcs}
Let $\A$ be a hyperplane arrangement, let $\M=\M(\A)$ be its associated matroid, 
and let $G=G(\A)$ be the arrangement group. 
Given a field $\k$ of characteristic $p>0$ and an integer $r>3$, 
determine whether the $\k$-linear maps 
\begin{align*}
&\begin{tikzcd}[ampersand replacement=\&, column sep=20pt]
\Psi \otimes \k \colon \h_r(\M) \otimes \k \ar[r, two heads] \& \gr_r(G)\otimes \k \end{tikzcd}\\[-3pt]
&\begin{tikzcd}[ampersand replacement=\&, column sep=20pt]\overline{\Psi}  \otimes \k \colon \big(\h(\M)/\h(\M)''\big)_r  \otimes \k  \ar[r, two heads] \&
\gr_r(G/G'') \otimes \k \end{tikzcd}
\end{align*}
are isomorphisms. 
\end{problem}

If the arrangement $\A$ is decomposable over $\Z$, then the maps 
$\Psi$ and $\overline{\Psi}$ are isomorphisms and all the Lie algebras 
involved are torsion-free, see \cite{PS-cmh06}. Furthermore, 
if $\A$ is decomposable over a field $\k$, similar proofs 
show that the maps $\Psi\otimes\k$ and $\overline{\Psi}\otimes\k$ 
are both isomorphisms. In general, though, these maps are not 
isomorphisms, as the next example illustrates.

\begin{example}
\label{ex:non-fano}
Let $\A$ be the non-Fano arrangement in $\C^3$, with defining polynomial 
$Q=xyz(x-y)(x-z)(y-z)(x+y-z)$. 
As noted in \cite{Su-conm01}, the first few LCS quotients are torsion-free, 
of ranks $\phi_1=7$, $\phi_2=6$, $\phi_3=17$, $\phi_4=42$, etc, 
while the Chen groups are all torsion-free, of ranks $\theta_{r}=9(r-1)$ 
for $r\ge 4$. On the other hand, the map $\Psi\otimes \Z_2$ 
has kernel equal to $\Z_2$ in degree $4$, generated by the quadruple Lie bracket 
$[x_7, [x_6 , [x_5, x_4]]]$, where $x_i$ is the generator  corresponding 
to the $i$-th hyperplane. Furthermore, the map $\overline{\Psi}\otimes \Z_2$ 
has also kernel equal to $\Z_2$ in degree $4$.
\end{example}

In \cite{Su-conm01}, we gave examples of arrangements $\A$ for which the 
associated graded Lie algebra $\gr(G(\A))$ may have non-trivial torsion. 
In view of this phenomenon, we asked whether such torsion is combinatorially determined. 
The question was answered in the negative by Artal Bartolo, Guerville-Ball\'e, 
and Viu-Sos \cite{AGV}, who constructed a pair $\A^{+}$ and $\A^{-}$ 
of lattice-isomorphic arrangements of $13$ hyperplanes in $\C^3$ 
for which $\gr_4(G(\A^+))\not\cong \gr_4(G(\A^-))$, with the 
difference between these two groups being detected by their $2$-torsion subgroups. 

\subsection{Resonance varieties of arrangements}
\label{subsec:res-hyparr}

The resonance varieties of an arrangement complement 
(over the coefficient field $\k=\C$) 
were first defined and studied by Falk in \cite{Fa97}.  An intense 
period of activity followed, in which the structure of these varieties 
was investigated in works such as \cite{CS-camb99, LY00, EPY03, DPS-duke, 
PY08, Yu09, Bu11, DeS-plms, De16, DSY17}, and extensions to coefficient fields of 
positive characteristic were explored in \cite{MS-aspm, Fa07, PS-plms17}. We review 
here some of those results. 

Let $\A$ be a complex hyperplane arrangement, with complement $M=M(\A)$. 
The resonance varieties of $M$ (over a coefficient field $\k$) 
are defined as the resonance varieties of the Orlik--Solomon $\k$-algebra of the 
corresponding matroid $\M=\M(\A)$; that is, $\RR^q_s(M,\k)\coloneqq \RR^q_s(\M,\k)$.  
From the general theory reviewed in \S\ref{subsec:res cga}, we know that each of 
these sets is a homogeneous subvariety of the affine space $V_{\k}=\OS^1(\M)\cong \k^n$, 
where $n=\abs{\A}$. 
In fact, if we fix a linear order on the hyperplanes, $\A=\{H_1,\dots, H_n\}$, 
all the resonance varieties lie in the hyperplane 
\begin{equation}
\label{eq:Vk}
\overline{V}_{\k} =\{ x \in \k^n : x_1 + \cdots + x_n=0\} . 
\end{equation}

The description of the Orlik--Solomon algebra given in 
\S\ref{subsec:OS} makes it clear that the resonance 
varieties $\RR^q_s(M(\A),\k)$ depend only on the 
intersection lattice, $L(\A)$, and on the characteristic 
of the field $\k$.  A basic problem in the subject is to 
find concrete formulas making this dependence explicit. 
One important piece of knowledge in this direction is the following result, 
which is a direct consequence of the formality of arrangement complements 
and of the Tangent Cone theorem of Dimca, Papadima, and Suciu 
\cite{DPS-duke} (in degree $q=1$) and Dimca and Papadima \cite{DP-ccm} 
(in arbitrary degrees). 

\begin{theorem}
\label{thm:res-arr-linear}
For any hyperplane arrangement $\A$, all the resonance varieties 
$\RR^q_s(M(\A),\C)$ are finite unions of rationally defined linear subspaces 
of $H^1(M(\A);\C)$.
\end{theorem}

Alternative proofs of this theorem in degree $q=1$ 
were first given in \cite{CS-camb99} and \cite{LY00}, 
using very different methods; see \S\ref{subsec:res-hyparr-deg1} 
for more on this. All these proofs use in an essential way geometric and topological 
methods, so this naturally raises the question whether the resonance 
varieties of non-realizable matroids are linear (in characteristic $0$). 
In degree $q=1$, we saw in Theorem \ref{thm:res fy} that this is indeed the case, 
but in higher degrees the question remains wide open.

\begin{question}
\label{quest:res-mat-lin}
Let $\M$ be a non-realizable simple matroid. If $q\ge 2$, are all the 
irreducible components of $\RR^q(\M,\C)$ (rationally defined) linear subspaces?
\end{question}

As highlighted in the works of Denham, Suciu, and Yuzvinsky \cite{DSY16, DSY17}, 
arrangement complements enjoy an important topological property 
that has a variety of implications regarding their homology groups and their 
cohomology jump loci. Let $X$ be a space having the homotopy type of 
a connected, finite CW-complex of dimension $d$, and let $G=\pi_1(X)$. 
We say that $X$ is an {\em abelian duality space}\/ of dimension $d$ if 
$H^q(X;\Z[G_{\ab}])=0$ for $q\ne d$ and $H^d(X;\Z[G_{\ab}])$ is non-zero 
and torsion-free. 

\begin{theorem}[\cite{DSY16, DSY17}]
\label{thm:arr-prop}
Let $\A$ be a central arrangement of rank $\ell$. Then:
\begin{enumerate}[label=(\arabic*)] 
\item \label{ab1}
The complement $M=M(\A)$ is an abelian duality space of 
dimension $\ell$ and the projectivized complement 
$U=\P(M)$ is an abelian duality space of dimension $\ell-1$.
\item \label{ab2}
$\RR^1(M,\k)\subseteq  \cdots\subseteq\RR^{\ell}(M,\k)$ 
and  
$\RR^1(U,\k)\subseteq  \cdots\subseteq\RR^{\ell-1}(U,\k)$, 
for any field $\k$.
\end{enumerate}
\end{theorem}

The result from \ref{r2} recovers by means 
of topological considerations the result of \cite{EPY03}  
regarding propagation of resonance for OS-algebras (see Corollary \ref{cor:res-mat}).
The diffeomorphism $M \cong U\times \C^*$, together with the 
product formula from Proposition \ref{prop:resprod} yields an identification 
$\RR^q(M,\k)\cong \RR^q(U,\k)\cup \RR^{q-1}(U,\k)$ for all $q\ge 1$.  
Applying Theorem \ref{thm:arr-prop}, part \ref{r2}, we conclude that 
$\RR^q(M,\k)\cong \RR^q(U,\k)$ for all $q\ge 1$.

We now return to the linear upper envelope for resonance described in 
Theorem \ref{thm:res-bound}.  As shown by Denham in \cite[Theorem~6]{De16}, 
this bound can be sharpened in the realizable case, as follows.  
Given a matroid $\M$ on ground set $[n]$, we say that  
a subset of flats $\CC\subset L(\M)$ covers $\M$ if
there is a surjective function $f\colon[n]\to\CC$ for which $i\in f(i)$ for
all $1\leq i\le n$.  

\begin{theorem}[\cite{De16}]
\label{thm:arr-res-bound}
Let $\A$ be an arrangement of $n$ hyperplanes in $\C^{\ell}$, with complement $M=M(\A)$ 
and associated matroid $\M=\M(\A)$. Then, for all fields $\k$ and all $q\ge 0$,
\[
\RR^q(M,\k) \subseteq \cS^q(\M, \k) \coloneqq \bigcup_{\CC} \bigcap_{X\in\CC} P_{\{X,[n]\setminus X\}} ,
\]
where $P_{\pi}$  is the subspace from \eqref{eq:P-pi} corresponding to a partition 
$\pi$ of $[n]$ and the union is over all subsets $\CC\subseteq L^{\irr}_{\le q+1}(\M)$ 
that cover $\M$.
\end{theorem}

Note that the upper envelope $\cS^q(\M, \k)$ is the union of an arrangement 
of linear subspaces in $\overline{V}_{\k}$ which is completely determined 
by the intersection lattice $L(\A)=L(\M)$. Nevertheless, the proof of 
Theorem \ref{thm:arr-res-bound} depends on the main result of \cite{CDO03}, 
which requires the matroid $\M$ to have a complex realization. So this 
naturally raises the question whether this hypothesis can be dispensed with.

\begin{question}[\cite{De16}]
\label{quest:better-res-bound}
Do the inclusions $\RR^q(\M,\k) \subseteq \cS^q(\M, \k)$ hold for all 
matroids $\M$, all fields $\k$, and all degrees $q\ge 2$?
\end{question}

To recap, this section has explored how properties such as linearity, propagation, 
and upper bounds constrain the resonance varieties of a matroid, rooted in its 
combinatorial structure. Yet, many of these results rely on the geometric realization 
of matroids as hyperplane arrangements, raising open questions about their validity 
for non-realizable matroids.  
This sets the stage for the next section, where the degree-$1$ resonance varieties 
$\RR^1_s(\M, \C)$ offer an illuminating case study, with the combinatorial data 
of multinets aligning closely with the geometric picture of arrangements.

\subsection{Resonance in degree $1$}
\label{subsec:res-hyparr-deg1}
Soon after the complex resonance varieties of an arrangement complement 
were introduced by Falk in \cite{Fa97}, the geometry of the varieties 
$\RR_s(\A)=\RR^1_s(M(\A),\C)$ was described in the work of 
Cohen--Suciu \cite{CS-camb99} and Libgober--Yuzvinsky \cite{LY00}. 
It was shown that $\RR_1(\A)$ consists of linear subspaces 
of the vector space $\overline{V}_{\C}\subset H^1(M;\C)$; 
that each irreducible component of $\RR_1(\A)$
is either $\{0\}$, or a linear subspace of $\overline{V}_{\C}$ 
of dimension at least $2$; and that two distinct components 
meet only at $0$. Moreover, $\RR_s(\A)$ is 
either $\{0\}$, or the union of all components of $\RR_1(\A)$ 
of dimension greater than $s$. 

The work of Falk and Yuzvinsky \cite{FY07} further clarified the structure of the 
variety $\RR_1(\A)$, by showing that the positive-dimensional 
components are in bijection with the multinets on sub-arrangements of $\A$.  
Although the same is true for arbitrary matroids (see Theorem \ref{thm:res fy}), 
the correspondence between resonance 
components and multinets can be described in geometric terms in the 
realizable case, as follows.

Let $\A$ be an arrangement in $\C^3$ with defining polynomial 
$f=\prod_{H\in \A} f_H$ and complement $M$. Given a $(k,d)$-multinet 
$\NN$ on $\A$, with parts $\A_i$ and multiplicity function $m\colon \A\to \N$, 
$H\mapsto m_H$, 
write $f_{i}=\prod_{H\in\A_{i}}f_H^{m_H}$ and define a rational map 
$\psi \colon \C^3\to\CP^1$ by $\psi(x)=\pcoor{f_1(x) : f_2(x)}$. 
There is then a set $D=\{ \pcoor{a_1 : b_1}, \dots , \pcoor{a_k : b_k}\}$ 
of $k$ distinct points in $\CP^1$
such that each of the degree $d$ polynomials $f_1,\dots, f_k$ 
can be written as $f_{i}=a_{i} f_2-b_{i}f_1$, and, 
furthermore, the image of $\psi\colon M\to\CP^1$ misses $D$.  
The map 
\begin{equation}
\label{eq:pencil}
\psi=\psi_{\NN} \colon M \to \CP^1\setminus D 
\end{equation}
is an orbifold fibration, known as the {\em pencil}\/ associated to the multinet $\NN$. 
Following \cite{PS-plms17, Su-toul}, we may describe the homomorphism 
induced in homology by this map, as follows. Let $\alpha_1,\dots ,\alpha_k$ 
be compatibly oriented simple closed curves on $S=\CP^1\setminus D$, going 
around the points of $D$, so that $H_1(S;\Z)$ is generated by the homology 
classes $c_{i}=[\alpha_{i}]$, subject to the single relation $\sum_{i=1}^k c_{i}=0$.   
Then the induced homomorphism $\psi_{*} \colon H_1(M;\Z) \to H_1(S;\Z)$ 
is given by $\psi_*(x_H) = m_H c_{i}$ for $H\in \A_{i}$, 
and thus $\psi^{*} \colon H^1(S;\Z) \to H^1(M;\Z)$ 
is given by $\psi^*(c_{i}^{\vee}) = w_i$, where 
$c_{i}^{\vee}$ is the Kronecker dual of $c_i$ 
and $w_i=\sum_{H\in \A_{i}} m_H e_H$.

It follows from the above discussion that the map 
$\psi^{*} \colon H^1(S;\C) \to H^1(M;\C)$ 
is injective, and thus sends $\RR^1_1(S)$ to $\RR^1_1(M)$.
Let us identify $\RR^1_1(S)$ with $H^1(S;\C)=\C^{k-1}$, and view 
$P_{\NN}\coloneqq \psi^*(H^1(S;\C))$ as lying inside 
$\RR_1(\A)\coloneqq \RR^1_1(M)$.  
Then $P_{\NN}$ is the $(k-1)$-dimensional linear 
subspace from \eqref{eq:pm}, 
spanned by the vectors $w_2-w_1,\dots , w_k-w_1$.  
Moreover, as shown in \cite[Theorems~2.4--2.5]{FY07}, this subspace is 
an essential component of $\RR_1(\A)$; that is, $P_{\NN}$ is not 
contained in any proper coordinate subspace of $H^1(M;\C)$.

\begin{remark}
\label{rem:bundle}
As shown in \cite[Corollary 4.3]{FY07}, equality holds in formula \eqref{eq:RH}  
if and only if the blocks of the multinet $\NN$ form all the singular fibers of the 
(projectivized) map $\bar\psi\colon \CP^2\to \CP^1$. In this case, the restriction 
$\bar\psi\colon \P(M)\to S=\CP^1\setminus \{\text{$k$ points}\}$ 
is a smooth fiber bundle with fiber a Riemann surface with some punctures, and 
therefore $\P(M)$ is aspherical.
\end{remark}

More generally, suppose there is a sub-arrangement 
$\BB\subseteq \A$ supporting a multinet $\NN$.   In this case, 
the inclusion $M(\A) \inj M(\BB)$ induces a 
monomorphism $H^1(M(\BB);\C) \inj H^1(M(\A);\C)$, 
which restricts to an embedding $\RR_1(\BB) \inj \RR_1(\A)$.  
The linear space $P_{\NN}$, then, 
lies inside $\RR_1(\BB)$, and thus, inside $\RR_1(\A)$.
Conversely, as shown in \cite[Theorem~2.5]{FY07} 
all (positive-dimensional) irreducible components 
of $\RR_1(\A)$ arise in this fashion. 

The next theorem, which combines results of Pereira--Yuzvinsky 
\cite{PY08} and Yuzvinsky \cite{Yu09}, summarizes what is known 
about the existence of non-trivial multinets on arrangements. 

\begin{theorem}[\cite{PY08, Yu09}] 
\label{thm:pyy}
Let $\NN$ be a $k$-multinet on an arrangement $\A$, 
with multiplicity function $m$ and base locus $\XX$. 
If $\abs{\XX}>1$, then $k=3$ or $4$.    
Furthermore, if there is a hyperplane $H\in \A$ such that $m_H>1$, 
then $k=3$. 
\end{theorem}

Although infinite families of multinets with $k=3$ are known, 
only one multinet with $k=4$ is known to exist: the $(4,3)$-net 
on the Hessian arrangement.  This leads to the following conjecture 
proposed by Yuzvinsky in \cite{Yu12}.

\begin{conjecture}[\cite{Yu12}]
\label{conj:hessian-net}
The only $(4,d)$-multinet on a hyperplane arrangement is the $(4,3)$-net on the 
Hessian arrangement.
\end{conjecture}

As mentioned by Yuzvinsky, the conjecture has been verified for $d=4,5,6$. 

\subsection{Concluding remarks}
\label{subsec:concluding}
This survey has sought to illuminate the rich interplay between the combinatorial 
structure of matroids and the topology of their realizations as hyperplane 
arrangements through algebraic invariants like the Orlik--Solomon algebra and 
the holonomy and Chen Lie algebras, and algebro-geometric invariants such 
as resonance varieties and Koszul modules. Insights from multinets offer 
a promising combinatorial lens for both realizable and non-realizable matroids. 
The scarcity of $k$-multinets with $k \geq 4$ in realizable cases, exemplified 
by the unique $(4,3)$-net on the Hessian arrangement, underscores their role 
in shaping resonance varieties. Open questions, such as the linearity of 
higher-degree resonance varieties over $\C$ and the structure of Koszul 
modules, persist. Additionally, the presence of torsion in the holonomy Lie 
algebra $\h(\M)$ of a matroid $\M$ or the associated graded Lie algebra 
$\gr(G)$ of an arrangement group may distinguish realizable 
from non-realizable matroids. These challenges invite continued exploration, 
promising deeper insights into the combinatorial-topological nexus of matroid theory.


\vskip 0,65 true cm

\noindent\textbf{Acknowledgements} \textit{Partially supported by the Simons 
Foundation Collaboration Grant for Mathematicians \#693825 and 
by the project ``Singularities and Applications" - CF 132/31.07.2023 funded by
the European Union - NextGenerationEU - through Romania’s National Recovery 
and Resilience Plan.}

\vskip 0,65 true cm

\medskip

\smallskip\noindent
Received:  26.07.2025\\
Accepted: 27.08.2025

\salt

\adresa{$^{(1)}$ Department of Mathematics, Northeastern University, Boston, MA 02115, USA\\
E-mail: {\tt a.suciu\ap northeastern.edu} }


\begin{thebibliography}{999} 

\bibitem{AFRSS24}
M.~Aprodu,  G.~Farkas, C.~Raicu, A.~Sammartano, A.I.~Suciu, 
\href{https://doi.org/10.1007/s10801-024-01313-2}%
{\em Higher resonance schemes and Koszul modules of simplicial 
complexes}, J.~Algebraic Combin. \textbf{59} (2024), no.~4, 787--805.

\bibitem{AFRS24} M.~Aprodu, G.~Farkas, C.~Raicu, A.I.~Suciu,
\href{https://doi.org/10.1515/crelle-2024-0051}%
{\em Reduced resonance schemes and Chen ranks}, 
J. Reine Angew. Math. \textbf{814} (2024), 205--240.

\bibitem{AFRS} M.~Aprodu, G.~Farkas, C.~Raicu, A.I.~Suciu,
{\em An effective proof of the Chen ranks conjecture},
preprint (2025).

\bibitem{AGV} E.~Artal Bartolo, B.~Guerville-Ball\'e, J.~Viu-Sos, 
\href{http://doi.org/10.1080/10586458.2018.1428131}%
{\em Fundamental groups of real arrangements and torsion in 
the lower central series quotients}, 
Exp. Math. \textbf{29} (2020), no.~1, 28--35. 

\bibitem{BZ}  A.~Bj\"orner, G.M.~Ziegler, 
\href{https://doi.org/10.1016/0095-8956(91)90008-8}%
{\em Broken circuit complexes: factorizations and generalizations}, 
J. Combin. Theory Ser. B \textbf{51} (1991), no.~1, 96--126.

\bibitem{Br}  E.~Brieskorn, 
\href{https://doi.org/10.1007/BFb0069274}%
{\em Sur les groupes de tresses (d'apr\`es V. I. Arnol'd)},  
S\'eminaire Bourbaki, 24\`eme ann\'ee (1971/1972), Exp. No. 401, 
pp. 21--44, Lecture Notes in Math., vol.~317, Springer, Berlin, 1973. 

\bibitem{Bu11}  N.~Budur, 
\href{https://doi.org/10.4310/MRL.2011.v18.n5.a5}%
{\em Complements and higher resonance varieties of 
hyperplane arrangements}, Math. Res. Lett.  \textbf{18} 
(2011), no.~5, 859--873. 

\bibitem{Chen51} K.-T.~Chen, 
\href{https://dx.doi.org/10.2307/1969316}%
{\em Integration in free groups}, Ann. of Math. (2) \textbf{54} 
(1951), no.~1, 147--162.

\bibitem{Chen73} K.-T.~Chen, 
\href{http://dx.doi.org/10.2307/1970846}%
{\em Iterated integrals of differential forms and loop space homology}, 
Ann. of Math. \textbf{97} (1973), 217--246. 

\bibitem{CDO03}  D.C.~Cohen, A.~Dimca, P.~Orlik,
\href{https://doi.org/10.5802/aif.1994}%
{\em Nonresonance conditions for arrangements}, 
Ann. Inst. Fourier (Grenoble) \textbf{53} (2003), 
no.~6, 1883--1896. 

\bibitem{CSc-15} D.C.~Cohen, H.K.~Schenck,
\href{https://dx.doi.org/10.1016/j.aim.2015.07.023}%
{\em Chen ranks and resonance}, Adv. Math. 
\textbf{285} (2015), 1--27.

\bibitem{CS-conm95} D.C.~Cohen, A.I.~Suciu, 
\href{http://dx.doi.org/10.1090/conm/181/02029}%
{\emph{The {C}hen groups of the pure braid group}}, 
in: {\em The \v {C}ech centennial} ({B}oston, {MA}, 1993), 45--64, 
Contemp. Math., vol.~181, Amer. Math. Soc., Providence, RI, 1995.

\bibitem{CS-cmh97}  D.C.~Cohen, A.I.~Suciu,
\href{http://dx.doi.org/10.1007/s000140050017}%
{\em The braid monodromy of plane algebraic curves and 
hyperplane arrangements}, Comment. Math. Helvetici 
\textbf{72} (1997), no.~2, 285--315. 

\bibitem{CS-tams99} D.C.~Cohen, A.I.~Suciu, 
\href{http://dx.doi.org/10.1090/S0002-9947-99-02206-0}%
{\em Alexander invariants of complex hyperplane arrangements}, 
Trans. Amer. Math. Soc. \textbf{351} (1999), no.~10, 4043--4067. 

\bibitem{CS-camb99} D.C.~Cohen, A.I.~Suciu, 
\href{http://dx.doi.org/10.1017/S0305004199003576}%
{\em Characteristic varieties of arrangements},
Math. Proc. Cambridge Phil. Soc. \textbf{127} (1999), 
no.~1, 33--53. 

\bibitem{DF17} E.~Delucchi, M.~Falk, 
\href{https://doi.org/10.1090/proc/13328}%
{\em An equivariant discrete model for complexified arrangement complements}, 
Proc. Amer. Math. Soc. \textbf{145} (2017), no.~3, 955--970.

\bibitem{De10} G.~Denham,  
\href{https://doi.org/10.1007/978-3-0346-0209-9_2}%
{\em Homological aspects of hyperplane arrangements}, 
in:  {\em Arrangements, local systems and singularities}, 
39--58, Progress in Math., vol.~283, Birkh\"{a}user, Basel, 2010. 

\bibitem{De16} G.~Denham, 
\href{https://doi.org/10.1007/978-3-319-31580-5_2}%
{\em Higher resonance varieties of matroids},  
 in: {\em Configurations Spaces}, 41--66, Springer 
 INdAM series, vol.~14, Springer, Cham, 2016.

\bibitem{DeS06}  G.~Denham, A.I.~Suciu, 
\href{http://dx.doi.org/10.1307/mmj/1156345597}%
{\em On the homotopy {L}ie algebra of an arrangement}  
Michigan Math. J. \textbf{54} (2006), no.~2, 319--340.  

\bibitem{DeS-plms} G.~Denham, A.I.~Suciu,
\href{https://doi.org/10.1112/plms/pdt058}%
{\em Multinets, parallel connections, and {M}ilnor fibrations  
of arrangements}, Proc. London Math. Soc. 
\textbf{108} (2014), no.~6, 1435--1470.

\bibitem{DSY16} G.~Denham, A.I.~Suciu, S.~Yuzvinsky, 
\href{https://dx.doi.org/10.1007/s00029-015-0196-8}%
{\em Combinatorial covers and vanishing of cohomology}, 
Selecta Math. (N.S.) \textbf{22} (2016), no.~2, 561--594. 

\bibitem{DSY17} G.~Denham, A.I.~Suciu, S.~Yuzvinsky, 
\href{https://dx.doi.org/10.1007/s00029-017-0343-5}%
{\em Abelian duality and propagation of resonance}, 
Selecta Math. \textbf{23} (2017), no.~4, 2331--2367. 

\bibitem{DP-ccm} A.~Dimca, S.~Papadima,
\href{http://dx.doi.org/10.1142/S0219199713500259}%
{\em Non-abelian cohomology jump loci from an analytic viewpoint}, 
Commun. Contemp. Math. \textbf{16} (2014), 
no.~4, 1350025 (47 p). 

\bibitem{DPS-duke} A.~Dimca, S.~Papadima, A.I.~Suciu,
\href{http://dx.doi.org/10.1215/00127094-2009-030}%
{\em Topology and geometry of cohomology jump loci}, 
Duke Math. Journal \textbf{148} (2009), no.~3, 405--457.

\bibitem{EPY03} D.~Eisenbud, S.~Popescu, S.~Yuzvinsky, 
\href{https://doi.org/10.1090/S0002-9947-03-03292-6}%
{\em Hyperplane  arrangement cohomology and monomials 
in the exterior algebra}, Trans. Amer. Math. Soc. \textbf{355} 
(2003), no.~11, 4365--4383.

\bibitem{Fa89} M.~Falk,
\href{https://doi.org/10.1090/conm/090/1000594}%
{\em The cohomology and fundamental group of a hyperplane
complement}, in: Singularities (Iowa City, IA, 1986), Contemp. Math.,
vol.~90, Amer. Math. Soc, Providence, RI, 1989, pp. 55--72.  

\bibitem{Fa97} M.~Falk,
\href{https://doi.org/10.1007/BF02558471}%
{\em Arrangements and cohomology},
Ann. Combin. \textbf{1} (1997), no.~2, 135--157.  

\bibitem{Fa07} M.~Falk, 
\href{https://doi.org/10.1093/imrn/rnm009}%
{\em Resonance varieties over fields of positive characteristic},
Int. Math. Research Notices \textbf{2007} (2007), no.~3, 
article ID rnm009, 25 pages.

\bibitem{FR85} M.~Falk, R.~Randell,
\href{https://doi.org/10.1007/BF01394780}%
{\em The lower central series of a fiber-type arrangement},
Invent. Math. \textbf{82} (1985), no.~1, 77--88. 

\bibitem{FY07} M.~Falk, S.~Yuzvinsky,
\href{https://dx.doi.org/10.1112/S0010437X07002722}%
{\em Multinets, resonance varieties, and pencils of plane curves},
Compositio Math. \textbf{143} (2007), no.~4, 1069--1088.

\bibitem{Ka07}  Y.~Kawahara, 
\href{https://doi.org/10.3836/tjm/1184963658}%
{\em The non-vanishing cohomology of Orlik-Solomon 
algebras}, Tokyo J.~Math. \textbf{30} (2007), no.~1, 223--238.

\bibitem{Ko83} T.~Kohno,
\href{https://doi.org/10.1017/S0027763000020547}%
{\em On the holonomy {L}ie algebra and the nilpotent completion
of the fundamental group of the complement of hypersurfaces},
Nagoya Math. J. \textbf{92} (1983), 21--37.  

\bibitem{Ko85} T.~Kohno, 
\href{https://doi.org/10.1007/BF01394779}%
{\em S\'{e}rie de Poincar\'{e}-Koszul associ\'{e}e aux groupes de
tresses pures}, Invent. Math. \textbf{82} (1985), 57--75.

\bibitem{KNP14}  G.~Korchm\'aros, G.~Nagy, N.~Pace,
\href{https://doi.org/10.1007/s10801-013-0474-5}%
{\em $3$-nets realizing a group in a projective plane}, 
J. Algebraic Combin. \textbf{39} (2014), no.~4, 939--966.

\bibitem{LY00} A.~Libgober, S.~Yuzvinsky,
\href{https://dx.doi.org/10.1023/A:1001826010964}%
{\em Cohomology of {O}rlik--{S}olomon algebras and local 
systems},  Compositio Math. \textbf{21} (2000), no.~3, 337--361.  

\bibitem{LS09}  P.~Lima-Filho, H.~Schenck, 
\href{https://doi.org/10.1093/imrn/rnn163}%
{\em Holonomy Lie algebras and the LCS formula for graphic arrangements}, 
Intern. Math. Research Notices \textbf{2009} (2009), no.~8, 1421--1432.

\bibitem{Lofwall-86} C.~L\"ofwall,
\href{https://doi.org/10.1007/BFb0075468}%
{\em On the subalgebra generated by the one dimensional elements
in the {Y}oneda Ext-algebra}, in: {\em Algebra, algebraic topology and
their interactions}, 
Lecture Notes in Math, vol~1183,
Springer-Verlag, Berlin-Heidelberg-New York, 1986, pp.~291--338.

\bibitem{Lofwall-16}  C.~L\"{o}fwall, 
\href{https://doi.org/10.1080/00927872.2015.1100303}%
{\em Decomposition theorems for a generalization of the holonomy {L}ie 
algebra of an arrangement}, Comm. Algebra \textbf{44} (2016), no.~11, 4654--4663. 

\bibitem{Lofwall-20} C.~L\"{o}fwall,  
{\em The holonomy Lie algebra of a matroid}, \arxiv{2012.12044}.

\bibitem{MB09} M.A. Marco Buzun\'{a}riz,
\href{https://doi.org/10.1007/s00373-009-0863-7}%
{\em A description of the resonance variety of a line combinatorics via 
combinatorial pencils}, Graphs Combin. \textbf{25} (2009), no.~4, 469--488. 

\bibitem{MP92} M.~Markl, \c{S}.~Papadima,
\href{https://doi.org/10.5802/aif.1315}%
{\em Homotopy {L}ie algebras and fundamental groups via deformation theory}, 
Ann. Inst. Fourier (Grenoble) \textbf{42} (1992), no.~4, 905--935.

\bibitem{Mt06} D.~Matei,
\href{http://dx.doi.org/10.2969/aspm/04310205}%
{\em Massey products of complex hypersurface complements},
in: Singularity Theory and its Applications, 205--219, Advanced 
Studies in Pure Math., vol.~43, Math. Soc. Japan, Tokyo, 2006.

\bibitem{MS-aspm}  D.~Matei, A.I.~Suciu, 
\href{http://dx.doi.org/10.2969/ASPM/02710185}%
{\em Cohomology rings and nilpotent quotients of real and 
complex arrangements}, in: {\em Arrangements--Tokyo 1998}, 
185--215, Advanced Studies Pure Math., vol.~27 (2000), 
Kinokuniya, Tokyo, 2000. 

\bibitem{NP13} G.~Nagy, N.~Pace, 
\href{https://doi.org/10.1016/j.jcta.2013.06.002}%
{\em On small $3$-nets embedded in a projective plane over a field}, 
J. Combin. Theory Ser. A \textbf{120} (2013), no.~7, 1632--1641.

\bibitem{OS} P.~Orlik, L.~Solomon,
\href{https://doi.org/10.1007/BF01392549}%
{\em Combinatorics and topology of complements of
hyperplanes}, Invent. Math. \textbf{56} (1980), no.~2, 167--189. 

\bibitem{Oxley} J.G.~Oxley, 
\href{https://doi.org/10.1093/acprof:oso/9780198566946.001.0001}%
{\em Matroid theory}, Second edition, Oxf. Grad. Texts Math., vol.~21, 
Oxford University Press, Oxford, 2011. 

\bibitem{PS-imrn04} S.~Papadima, A.I.~Suciu, 
\href{http://dx.doi.org/10.1155/S1073792804132017}%
{\em Chen {L}ie algebras}, Int. Math. Res. Not. 
\textbf{2004} (2004), no.~21, 1057--1086. 

\bibitem{PS-cmh06} S.~Papadima, A.I.~Suciu,
\href{http://dx.doi.org/10.4171/CMH/77}%
{\em When does the associated graded Lie algebra of an
arrangement group decompose?}, 
Comment. Math. Helv. \textbf{81} 
(2006), no.~4, 859--875. 

\bibitem{PS-mrl} \c{S}.~Papadima, A.I.~Suciu,
\href{http://dx.doi.org/10.4310/MRL.2014.v21.n4.a13}%
{\em Jump loci in the equivariant spectral sequence},
Math. Res. Lett. \textbf{21} (2014), no.~4, 863--883.

\bibitem{PS-crelle} S.~Papadima, A.I.~Suciu,
\href{http://dx.doi.org/10.1515/crelle-2013-0073}%
{\em Vanishing resonance and representations of Lie algebras}, 
J. Reine Angew. Math. \textbf{706} (2015), 83--101.  

\bibitem{PS-plms17} S.~Papadima, A.I.~Suciu,
\href{http://dx.doi.org/10.1112/plms.12027}
{\em The Milnor fibration of a hyperplane arrangement: from 
modular resonance to algebraic monodromy}, Proc. London 
Math. Soc.  \textbf{114} (2017), no.~6, 961--1004.

\bibitem{PaY}  S.~Papadima, S.~Yuzvinsky,
\href{https://doi.org/10.1016/S0022-4049(98)00058-9}%
{\em On rational $K[\pi,1]$ spaces and {K}oszul algebras},
J. Pure Appl. Alg. \textbf{144} (1999), 156--167.

\bibitem{PY08} J.~Pereira, S.~Yuzvinsky, 
\href{https://doi.org/10.1016/j.aim.2008.05.014}%
{\em Completely reducible hypersurfaces in a pencil}, 
Adv.  Math.  \textbf{219} (2008), no.~2, 672--688.

\bibitem{PP} A.~Polishchuk, L.~Positselski,
\href{http://dx.doi.org/10.1090/ulect/037}%
{\em Quadratic algebras}, University Lecture Series, vol.~37, 
American Math. Soc., Providence, RI, 2005. 

\bibitem{PS-ejm20} R.D.~Porter, A.I.~Suciu, 
\href{https://doi.org/10.1007/s40879-019-00392-x}%
{\em Homology, lower central series, and hyperplane arrangements},
Eur. J. Math. \textbf{6} (2020), nr.~3, 1039--1072.

\bibitem{Ra97} R.~Randell,
\href{https://doi.org/10.1016/S0166-8641(96)00123-X}
{\em Homotopy and group cohomology of arrangements},
Topology Appl. \textbf{78} (1997), 201--213.

\bibitem{SV91} V.V.~Schechtman, A.N.~Varchenko, 
\href{https://doi.org/10.1007/BF01243909}%
{\em Arrangements of hyperplanes and Lie algebra homology}, 
Invent. Math. \textbf{106} (1991), no.~1, 139--194.

\bibitem{SS-tams02} H.K.~Schenck, A.I.~Suciu, 
\href{http://dx.doi.org/10.1090/S0002-9947-02-03021-0}%
{\em Lower central series and free resolutions of hyperplane 
arrangements}, Trans. Amer. Math. Soc. \textbf{354} (2002), 
no.~9, 3409--3433. 

\bibitem{SS-tams05} H.K.~Schenck, A.I.~Suciu, 
\href{http://dx.doi.org/10.1090/S0002-9947-05-03853-5}%
{\em Resonance, linear syzygies, Chen groups, and the 
Bernstein--Gelfand--Gelfand correspondence}, 
Trans. Amer. Math. Soc. \textbf{358} (2005), 
no.~5, 2269--2289.

\bibitem{SY97} B.~Shelton, S.~Yuzvinsky,
\href{https://doi.org/10.1112/S0024610797005553}%
{\em {K}oszul algebras from graphs and hyperplane arrangements},
J. London Math. Soc. \textbf{56} (1997), 477--490.

\bibitem{Stanley} R.P. Stanley, 
\href{https://doi.org/10.1007/BF02945028}%
{\em Supersolvable lattices}, Algebra Universalis \textbf{2} (1972), 197--217. 

\bibitem{Su-conm01} A.I.~Suciu,
\href{http://dx.doi.org/10.1090/conm/276/04510}%
{\em Fundamental groups of line arrangements: 
Enumerative aspects}, in: {\em Advances in algebraic geometry 
motivated by physics (Lowell, MA, 2000)}, 43--79, Contemp. 
Math., vol 276, Amer. Math. Soc., Providence, RI, 2001.

\bibitem{Su-toul} A.I.~Suciu,
\href{http://dx.doi.org/10.5802/afst.1412}%
{\em Hyperplane arrangements and {M}ilnor fibrations}, 
Ann. Fac. Sci. Toulouse Math. (6) \textbf{23} (2014), no.~2, 417--481.  

\bibitem{Su-tcone} A.I.~Suciu,
\href{http://dx.doi.org/dx.doi.org/10.1007/978-3-319-31580-5_1}
{\em Around the tangent cone theorem}, 
in: {\em Configuration Spaces: Geometry, 
Topology and Representation Theory}, 1--39, Springer INdAM 
series, vol.~14, Springer, Cham, 2016. 

\bibitem{Su-edinb20} A.I.~Suciu,
\href{https://doi.org/10.1017/prm.2019.55}%
{\em Poincar\'{e} duality and resonance varieties}, 
Proc. Roy. Soc. Edinburgh Sect. A. (2019), 
\textbf{150} (2020), nr.~6, 3001--3027. 

\bibitem{Su-conm23} A.I.~Suciu,
\href{https://doi.org/10.1090/conm/790/15861}%
{\em Cohomology, Bocksteins, and resonance varieties in 
characteristic $2$}, 131--157, Contemp. Math., 
vol.~790, Amer. Math. Soc., Providence, RI, 2023.

\bibitem{Su-pisa24} A.I.~Suciu,
\href{http://doi.org/10.2422/2036-2145.202112_005}%
{\em Alexander invariants and cohomology jump loci in group extensions}, 
Ann. Sc. Norm. Super. Pisa Cl. Sci. (5) \textbf{25} (2024), no.~2, 1085--1154.

\bibitem{Su-decomp} A.I.~Suciu, 
{\em On the topology and combinatorics of decomposable arrangements}, 
to appear in Contemp. Math., available at 
\arxiv{2404.04784}. 

\bibitem{SW-mz17} A.I.~Suciu, H.~Wang, 
\href{http://dx.doi.org/10.1007/s00209-016-1811-x}
{\em Pure virtual braids, resonance, and formality}, 
Math. Zeit. \textbf{286} (2017), no. 3-4, 1495--1524.   

\bibitem{SW-jpaa19}  A.I.~Suciu, H.~Wang, 
\href{https://doi.org/10.1016/j.jpaa.2018.11.006}
{\em Cup products, lower central series, and holonomy Lie algebras}, 
J. Pure Appl. Algebra  \textbf{223} (2019), no.~8, 3359--3385.

\bibitem{SW-forum} A.I.~Suciu, H.~Wang, 
\href{https://doi.org/10.1515/forum-2018-0098}%
{\em Formality properties of finitely generated groups and Lie 
algebras}, Forum Math. \textbf{31} (2019), no.~4, 867--905. 

\bibitem{Sullivan} D.~Sullivan, 
\href{https://dx.doi.org/10.1007/BF02684341}%
{\em Infinitesimal computations in topology}, Inst. Hautes \'Etudes 
Sci. Publ. Math. (1977), no.~47, 269--331. 

\bibitem{Uz} G.~Urz\'ua, 
\href{https://doi.org/10.1515/ADVGEOM.2010.006}%
{\em On line arrangements with applications to $3$-nets}, 
Adv. Geom. \textbf{10} (2010), no.~2, 287--310.

\bibitem{White} N.~White (ed.), 
\href{https://doi.org/10.1017/CBO9780511629563}%
{\em Theory of matroids}, Encyclopedia Math. Appl., vol.~26, 
Cambridge University Press, Cambridge, 1986. 

\bibitem{Wilson} R.J.~Wilson,
\href{https://doi.org/10.2307/2319608}%
{\em An introduction to matroid theory},
Amer. Math. Monthly \textbf{80} (1973), no.~5, 500--525. 

\bibitem{Yu95} S.~Yuzvinsky, 
\href{https://doi.org/10.1080/00927879508825535}%
{\em Cohomology of the Brieskorn--Orlik--Solomon algebras}, 
Comm. Algebra \textbf{23} (1995), no.~14, 5339--5354. 

\bibitem{Yu01} S.~Yuzvinsky, 
\href{https://doi.org/10.1070/RM2001v056n02ABEH000383}%
{\em Orlik–Solomon algebras in algebra and topology}, 
Russian Math. Surveys \textbf{56} (2001), no.~2, 293--364

\bibitem{Yu04} S.~Yuzvinsky, 
\href{https://doi.org/10.1112/S0010437X04000600}%
{\em Realization of finite abelian groups by nets in $\P^2$},  
Compos. Math. \textbf{140} (2004), no.~6, 1614--1624.

\bibitem{Yu09} S.~Yuzvinsky, 
\href{https://doi.org/10.1090/S0002-9939-08-09753-0}%
{\em A new bound on the number of special fibers in a 
pencil of curves}, Proc. Amer. Math. Soc. \textbf{137} (2009), 
no.~5, 1641--1648. 

\bibitem{Yu12}  S.~Yuzvinsky, 
\href{https://doi.org/10.2969/aspm/06210553}%
{\em Resonance varieties of arrangement complements}, 
in: {\em Arrangements of Hyperplanes (Sapporo 2009)},  
553--570, Advanced Studies Pure Math., vol.~62, Kinokuniya, 
Tokyo, 2012. 

\end{thebibliography}
\end{document}